\documentclass{amsart}
\usepackage[utf8]{inputenc}

\usepackage{graphics}
\usepackage{thmtools}
\usepackage{wasysym}

\usepackage[T1]{fontenc}    

\usepackage{amsthm}
\usepackage{amsbsy,amsmath,amssymb,amscd,amsfonts}
\usepackage{hyperref}
\usepackage{url}            
\usepackage{booktabs}       
\usepackage{nicefrac}       
\usepackage{microtype}      

\usepackage{graphicx,float,latexsym,color}
\usepackage[font={scriptsize,it}]{caption}
\usepackage{subcaption}

\usepackage{makecell}

\usepackage[dvipsnames]{xcolor}

\newtheorem{theorem}{Theorem}

\newtheorem{proposition}{Proposition}

\newtheorem{corollary}{Corollary}
\newtheorem{lemma}{Lemma}
\theoremstyle{remark}
\newtheorem{remark}{Remark}
\theoremstyle{definition}
\newtheorem{definition}{Definition}

\hypersetup{
    pdftoolbar=true,        
    pdfmenubar=true,        
    pdffitwindow=false,     
    pdfstartview={FitH},    
    colorlinks=true,       
    linkcolor=OliveGreen,          
    citecolor=blue,        
    filecolor=black,      
    urlcolor=red           
}

\usepackage{lineno}

\usepackage{comment}
\usepackage{microtype}

\newcommand{\E}{\mathcal{E}}
\newcommand{\K}{\mathcal{K}}
\newcommand{\T}{\mathcal{T}}
\renewcommand{\S}{\mathcal{S}}
\renewcommand{\H}{\mathcal{H}}
\renewcommand{\L}{\mathcal{L}}

\title[Steiner's Hat: a Constant-Area Deltoid]{Steiner's Hat: a Constant-Area Deltoid\\Associated with the Ellipse}
\author[R. Garcia]{Ronaldo Garcia}
\author[D. Reznik]{Dan Reznik}
\author[H. Stachel]{Hellmuth Stachel}
\author[M. Helman]{Mark Helman}
\date{May, 2020}

\begin{document}

\maketitle



\begin{abstract}
The Negative Pedal Curve (NPC) of the Ellipse with respect to a boundary point $M$ is a 3-cusp closed-curve which is the affine image of the Steiner Deltoid. Over all $M$ the family has invariant area and displays an array of interesting properties. 

\vskip .3cm
\noindent\textbf{Keywords} curve, envelope, ellipse, pedal, evolute, deltoid, Poncelet, osculating, orthologic.
\vskip .3cm
\noindent \textbf{MSC} {51M04 \and 51N20 \and 65D18}
\end{abstract}

\section{Introduction}
\label{sec:intro}
Given an ellipse $\E$ with non-zero semi-axes $a,b$ centered at $O$, let $M$ be a point in the plane. The Negative Pedal Curve (NPC) of $\E$ with respect to $M$ is the envelope of lines passing through points $P(t)$ on the boundary of $\E$ and perpendicular to $[P(t)-M]$ \cite[pp. 349]{stachel2019-conics}. Well-studied cases \cite{lockwood1961-curves,zwikker2005} include placing $M$ on (i) the major axis: the NPC is a two-cusp ``fish curve'' (or an asymmetric ovoid for low eccentricity of $\E$); (ii) at $O$: this yielding a four-cusp NPC known as Talbot's Curve (or a squashed ellipse for low eccentricity), Figure~\ref{fig:fish_talbot}.

\begin{figure}[H]
    \centering
    \includegraphics[width=.9\textwidth]{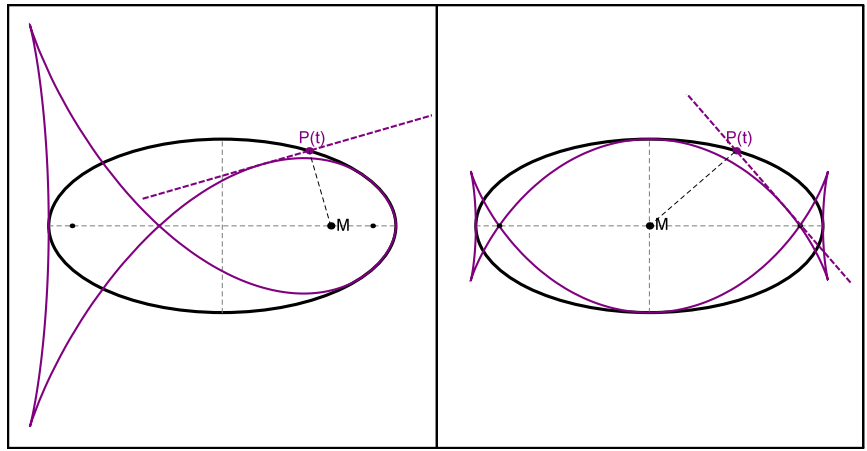}
    \caption{The Negative Pedal Curve (NPC) of an ellipse $\E$ with respect to a point $M$ on the plane is the envelope of lines passing through $P(t)$ on the boundary, and perpendicular to $P(t)-M$. \textbf{Left:} When $M$ lies on the major axis of $\E$, the NPC is a two-cusp ``fish'' curve. \textbf{Right:} When $M$ is at the center of $\E$, the NPC is 4-cusp curve with 2-self intersections known as Talbot's Curve \cite{mw}. For the particular aspect ratio $a/b=2$, the two self-intersections are at the foci of $\E$.}
    \label{fig:fish_talbot}
\end{figure}

As a variant to the above, we study the family of NPCs with respect to points $M$ on the {\em boundary} of $\E$. As shown in Figure~\ref{fig:main}, this yields a family of asymmetric, constant-area 3-cusped deltoids. We call these curves ``Steiner's Hat'' (or $\Delta$), since under a varying affine transformation, they are the image of the Steiner Curve (aka. Hypocycloid), Figure~\ref{fig:steiners}. Besides these remarks, we've observed:

\smallskip
\noindent \textbf{Main Results:}

\begin{itemize}
    \item The triangle $T'$ defined by the 3 cusps $P_i'$ has invariant area over $M$, Figure~\ref{fig:preimg_tri}.
    \item The triangle $T$ defined by the pre-images $P_i$ of the 3 cusps has invariant area over $M$, Figure~\ref{fig:preimg_tri}. The $P_i$ are the 3 points on $\E$ such that the corresponding tangent to the envelope is at a cusp.
    \item The $T$ are a Poncelet family with fixed barycenter; their caustic is half the size of $\E$, Figure~\ref{fig:preimg_tri}.
    \item Let $C_2$ be the center of area of $\Delta$. Then $M,C_2,P_1,P_2,P_3$ are concyclic, Figure~\ref{fig:preimg_tri}. The lines $P_i-C_2$ are tangents at the cusps.
    \item Each of the 3 circles passing through $M,P_i,P_i'$, $i=1,2,3$, osculate $\E$ at $P_i$, Figure~\ref{fig:osculating}. Their centers define an area-invariant triangle $T''$ which is a half-size homothety of $T'$. 
\end{itemize}

\begin{figure}
    \centering
    \includegraphics[width=.9\textwidth]{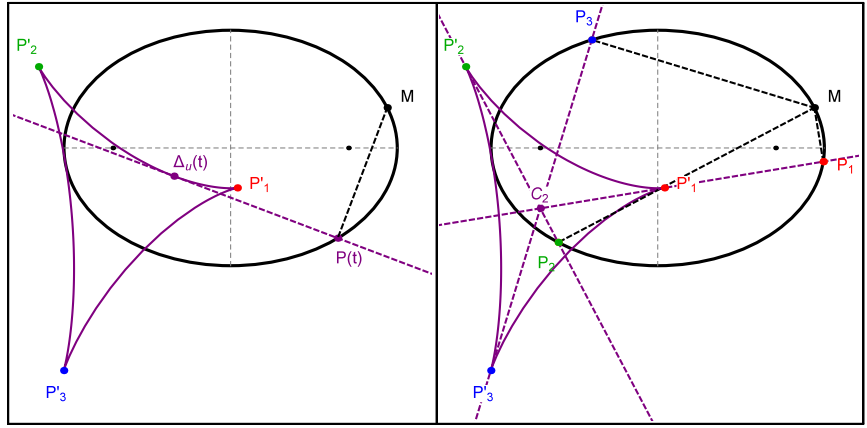}
    \caption{\textbf{Left:} The Negative Pedal Curve (NPC, purple) of $\E$ with respect to a boundary point $M$ is a 3-cusped (labeled $P_i'$) asymmetric curve (called here ``Steiner's Hat''), whose area is invariant over $M$, and whose asymmetric shape is affinely related to the Steiner Curve \cite{mw}. $\Delta_u(t)$ is the instantaneous tangency point to the Hat. \textbf{Right}: The tangents at the cusps points $P_i'$  
    concur at $C_2$, the Hat's center of area, furthermore, $P_i,P_i',C_2$ are collinear. \textbf{Video:} \cite[PL\#01]{playlist2020-deltoid}}
    \label{fig:main}
\end{figure}

\begin{figure}
    \centering
    \includegraphics[width=.9\textwidth]{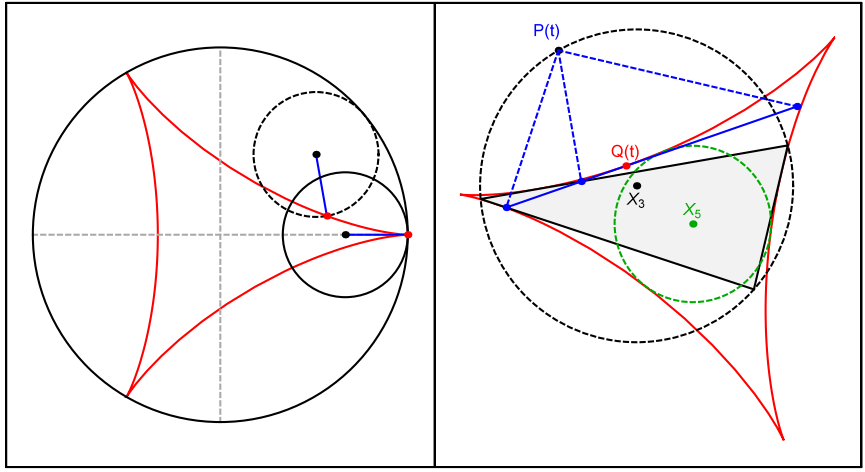}
    \caption{Two systems which generate the 3-cusp Steiner Curve (red), see \cite{ferreol2017-deltoid} for more methods. \textbf{Left:} The locus of a point on the boundary of a circle of radius $1 $ rolling inside another of radius $3$.
    \textbf{Right:} The envelope of Simson Lines (blue) of a triangle $T$ (black) with respect to points $P(t)$ on the Circumcircle \cite{mw}. $Q(t)$ denotes the corresponding 
    tangent. Nice properties include (i) the area of the Deltoid is half that of the Circumcircle, and (ii) the 9-point circle of $T$ (dashed green) centered on $X_5$ (whose radius is half that of the Circumcircle) is internally tangent to the Deltoid \cite[p.231]{wells1991}.}
    \label{fig:steiners}
\end{figure}


The paper is organized as follows. In Section~\ref{sec:main_props} we prove the main results. In Sections~\ref{sec:cusp_tri} and \ref{sec:preimg_tri} we describe properties of the triangles defined by the cusps and their pre-images, respectively. In Section~\ref{sec:cusp_locus} we analyze the locus of the cusps. In Section~\ref{sec:deltoid_tangs} we characterize the tangencies and intersections of Steiner's Hat with the ellipse. In Section~ \ref{sec:osc_circs} we describe properties of 3 circles which osculate the ellipse at the cusp pre-images and pass through $M$. In Section~\ref{sec:triangles} we describe relationships between the (constant-area) triangles with vertices at (i) cusps, (ii) cusp pre-images, and (iii) centers of osculating circles. In Section~\ref{sec:npc-rot} we analyze a fixed-area deltoid obtained from a ``rotated'' negative pedal curve. The paper concludes in Section~\ref{sec:conclusion} with a table of illustrative videos. Appendix~\ref{app:cusps} provides explicit coordinates for cusps, pre-images, and osculating circle centers. Finally, Appendix~\ref{app:symbols} lists all symbols used in the paper.

\section{Preliminaries}
\label{sec:preli}
Let the ellipse $\E$ be defined implicitly as:
 \[\E(x,y)=\frac{x^2}{a^2}+\frac{y^2}{b^2}-1=0, \;\;\; c^2=a^2-b^2\]

\noindent where $a>b>0$ are the semi-axes. Let a point $P(t)$ on its boundary be parametrized as $ P(t)=(a\cos t,b\sin t)$.

 Let $P_0=(x_0,y_0)\in\mathbb{R}^2$. Consider the family of lines $L(t)$ passing through $P(t)$ and orthogonal to $P(t)-P_0$.
 Its envelope $\Delta$ is called \textit{antipedal} or \textit{negative pedal curve} of $\E$.
 
 Consider the spatial curve defined by
 \[ \L(P_0)=\{(x,y, t): L(t,x,y)= L'(t,x,y)=0\}.\]
 The projection $\E(P_0)=\pi(\L(P_0))$ is the envelope. Here $\pi(x,y,t)=(x,y)$. In general, $\L(P_0)$ is regular, but $\E(P_0)$ is a curve with singularities and/or cusps.
 
 \begin{lemma}
 \label{lem:envelope}
 	The envelope of the family of lines $L(t)$ is given by:
 	\begin{align}
 	x(t)=&\frac{1}{w}[ ( a y_0\sin t -a b) x_0-b y_0^2\cos t  -c^2y_0\sin(2 t) \nonumber \\
 	+&\frac{b}{4}  (  (5a^2- b^2)\cos t-c^2\cos(3t) )] \nonumber\\
 	y(t)=&\frac{1}{w}[ - a x_0^2\sin t+( b y_0\cos t+c^2\sin(2 t)  ) x_0-aby_0 \nonumber\\
 	-&\frac{a}{4}  (  (5a^2- b^2)\sin t-c^2\sin(3t) ] 	 	\label{eq:envelope}
 	\end{align}
 where $w=ab- b x_0\cos t- a y_0\sin t$.
 \end{lemma}

\begin{proof} The line $L(t)$ is given by:
	\[ ( x_0-a\cos t) x+( y_0-b\sin t)y+a^2\cos^2t+b^2\sin^2t-ax_0\cos t-by_0\sin t =0 \]
Solving the linear system $L(t)=L^\prime(t)=0$ in the variables $x,y$ leads to the result.	
\end{proof}

Triangle centers will be identifed below as $X_k$ following Kimberling's Encyclopedia \cite{etc}, e.g., $X_1$ is the Incenter, $X_2$ Barycenter, etc.

\section{Main Results}
\label{sec:main_props}
\begin{proposition}
The NPC with respect to $M_u=(a\cos u, b\sin u)$ a boundary point of $\E$ is a 3-cusp closed curve given by
$\Delta_u(t)=(x_u(t),y_u(t))$, where  
\begin{align}
x_u(t)=&\frac{1}{a }\left(c^2  (1+\cos(t+u))\cos t -a^2 \cos u\right) \nonumber \\
y_u(t)=& \frac{1}{b}\left(c^2 \cos t \sin(t+u)-c^2\sin t-a^2\sin u\right)
\end{align}
\label{eq:deltoide}
\end{proposition}

\begin{proof} It is direct consequence of Lemma \ref{lem:envelope} with $P_0=M_u$.
\end{proof}

Expressions for the 3 cusps $P_i'$ in terms of $u$ appear in Appendix~\ref{app:cusps}.

\begin{remark}
As $a/b\rightarrow{1}$ the ellipse becomes a circle and $\Delta$ shrinks to a point on the boundary of said circle.
\end{remark}

\begin{remark}
 Though $\Delta$ can never have three-fold symmetry, for $M_u$ at any ellipse vertex, it has axial symmetry.
\end{remark}

\begin{remark}
The average coordinates $\bar{C}=[\bar{x}(u),\bar{y}(u)]$ of $\Delta_u$ w.r.t. this parametrization are given by:

\begin{align}
\bar{x}(u)&=\frac{1}{2\pi} \int_0^{2\pi} x_u(t)\,dt=-\frac{(a^2+b^2)}{2a}\cos u \nonumber\\ \bar{y}(u)&=\frac{1}{2\pi} \int_0^{2\pi} y_u(t)\,dt=-\frac{(a^2+b^2)}{2b}\sin u
\label{eqn:cbar}
\end{align}
\end{remark}

\begin{theorem}
 $\Delta_u$ is the image of the 3-cusp Steiner Hypocycloid $\S$ under a varying affine transformation.
\label{th:affine}
\end{theorem}

\begin{proof} Consider the following transformations in $\mathbb{R}^2$:

\[
\begin{array}{lcl}
\mbox{rotation:} & R_u(x,y)=&\begin{pmatrix} \cos{u} & \sin {u}\\ -\sin{u} & \cos{u} 
	 \end{pmatrix}  \begin{pmatrix} x \\
	 y \end{pmatrix} \\
\mbox{translation:} & U(x,y)=&(x,y)+\bar{C} \\
\mbox{homothety:} & D(x,y) =& 
\frac{1}{2}(a^2-b^2)(x,y).
\\ \mbox{linear map:} & V(x,y) =&(\frac{x}{a},\frac{y}{b})
\end{array}
\]

The hypocycloid of Steiner is given by $\S(t)=2(\cos t,-\sin t)+(\cos 2t, \sin 2t)$ \cite{lockwood1961-curves}. Then:

\begin{equation}
    \Delta_u(t)=[x_u(t),y_u(t)]=  (U\circ V\circ  D\circ R_u) \S(t)
    \label{eqn:affine}
\end{equation}

Thus, Steiner's Hat is of degree 4 and of class 3 (i.e., degree of its dual).
\end{proof}

\begin{corollary}
The area of $A(\Delta)$ of Steiner's Hat is invariant over $M_u$ and is given by:

\begin{equation}
A(\Delta) =\frac{(a^2-b^2)^2\pi}{2 a b}= \frac{c^4 \pi}{ 2 a b}
\label{eqn:adelta}
\end{equation}
\end{corollary}

\begin{proof}
The area of $\S(t)$ is $\int_S xdy=2\pi$.   The Jacobian of $ (U\circ S\circ  D\circ R_u)$ given by Equation~\ref{eqn:affine} is constant and equal to $c^4/4ab$.
\end{proof}

Noting that the area of $\E$ is ${\pi}{a}{b}$, Table~\ref{tab:area_ratios} shows the aspect ratios $a/b$ of $\E$ required to produce special area ratios.

\begin{table}[H]
\begin{tabular}{|r|c|l|}
\hline
$a/b$ & approx. $a/b$ &  $A(\Delta)/A(\E)$  \\
\hline
$\sqrt{2 + \sqrt{3}}$ & $1.93185$ & 1 \\
$\varphi=(1+\sqrt{5})/2$ & $1.61803$ & 1/2\\
$\sqrt{2}$ & $1.41421$ & 1/4 \\
$1$ & 1 & 0 \\
\hline
\end{tabular}
\caption{Aspect ratios yielding special area ratios of main ellipse $\E$ to Steiner's Hat $\Delta$.} 
\label{tab:area_ratios}
\end{table}

It is well known that if $M$ is interior to $\E$ then the NPC is a 2-cusp or 4-cusp curve with one or two self-intersections.

\begin{remark}
It can be shown that when $M$ is interior to $\E$ the iso-curves of signed area of the NPC are closed algebraic curves of degree 10, concentric with $\E$ and symmetric about both axes, see Figure~\ref{fig:iso-inner}.
\end{remark}

It is remarkable than when $M$ moves from the interior to the boundary of $\E$, the iso-curves transition from a degree-10 curve to a simple conic.

\begin{figure}
    \centering
    \includegraphics[width=\textwidth]{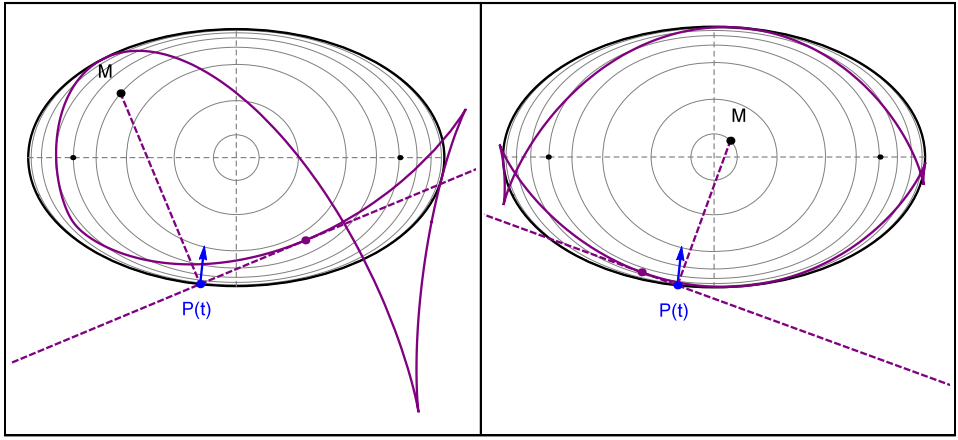}
    \caption{The isocurves of signed area for the negative pedal curve when $M$ is interior to the ellipse are closed algebraic curves of degree 10. These are shown in gray for an NPC with two cusps (left), and 4 cusps (right).}
    \label{fig:iso-inner}
\end{figure}

\begin{remark}
It can also be shown that when $M$ is exterior to $\E$, the NPC is a two-branched open curve, see Figure~\ref{fig:m-outside}.
\end{remark}

\begin{figure}
    \centering
    \includegraphics[width=\textwidth]{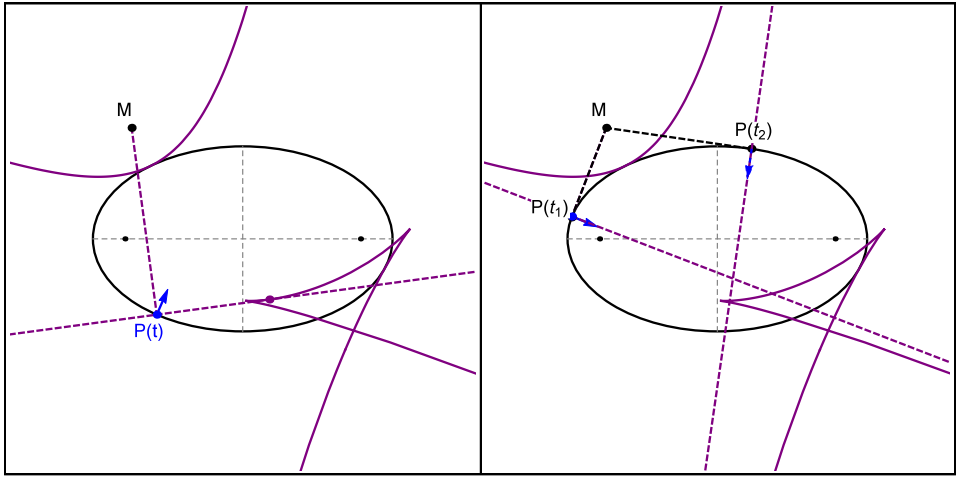}
    \caption{\textbf{Left:} When $M$ is exterior to $\E$ the NPC is a two-branched open curve. One branch is smooth and non-self-intersecting, and the other has 2 cusps and one self-intersection. \textbf{Right:} Let $t_1,t_2$ be the parameters for which $MP(t)$ is tangent to $\E$. At these positions, the NPC intersects the line at infinity in the direction of the normal at $P(t_1),P(t_2)$, i.e., the lines through $P(t)$ perpendicular to $P(t)-M$ are asymptotes.}
    \label{fig:m-outside}
\end{figure}

\begin{proposition}

Let $C_2$ be the center of area of $\Delta_u$. Then $C_2=\bar{C}$.
\end{proposition}

\begin{proof}
	
The center of area is defined by
\[ C_2=\frac{1}{A(\Delta)}\left(  \int_{int(\Delta)} x\,dx\,dy, \int_{int(\Delta)} y\,dx\,dy\right).\]

Using Green's Theorem, evaluate the above using the parametric in Equation~\eqref{eq:deltoide}. This yields the expression for $\bar{C}$ in Equation~\ref{eqn:cbar}. Alternatively, one can obtain the same result from the affine transformation defined in Theorem~\ref{th:affine}.
\end{proof}

\noindent Referring to Figure~\ref{fig:deltoid_c2}(left): 
\begin{corollary}
The locus of $C_2$ is an ellipse always exterior to a copy of $\E$ rotated $90^\circ$ about $O$.
\end{corollary}

\begin{figure}
    \centering
    \includegraphics[width=\textwidth]{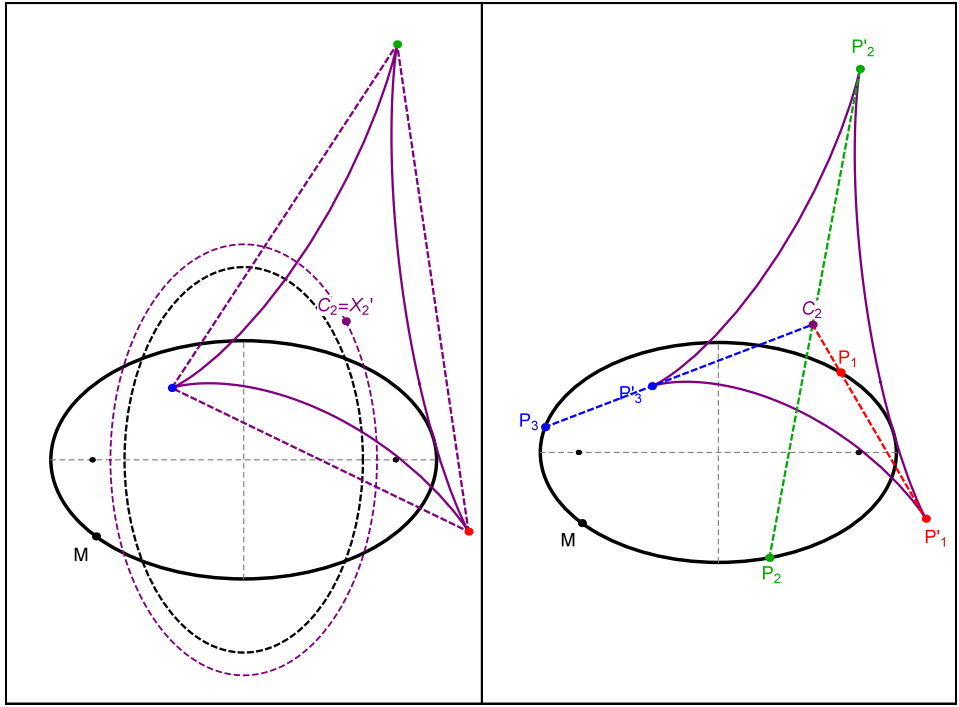}
    \caption{\textbf{Left:} The area center $C_2$ of Steiner's Hat coincides with the barycenter $X_2'$ of the (dashed) triangle $T'$ defined by the cusps. Over all $M$, both the Hat and $T'$ have invariant area. $C_2$'s locus (dashed purple) is elliptic and exterior to a copy of $\E$ rotated $90^\circ$ about its center (dashed black). \textbf{Right:} Let $P_i'$ (resp. $P_i$), $i=1,2,3$ denote the Hat's cusps (resp. their {\em pre-images} on $\E$), colored by $i$. Lines ${P_i}{P_i'}$ concur at $C_2$.}
    \label{fig:deltoid_c2}
\end{figure}

\begin{proof}
Equation~\ref{eqn:cbar} describes an ellipse. Since $a^2+b^2{\geq}2ab$ the claim follows directly.
\end{proof}

\noindent Let $T$ denote the triangle of the {\em pre-images} $P_i$ on $\E$ of the Hat's cusps, i.e. $P(t_i)$ such that $\Delta_u{t_i}$ is a cuspid. Explicit expressions for the $P_i$ appear in Appendix~\ref{app:cusps}. Referring to Figure~\ref{fig:preimg_tri}:

\begin{theorem}
  The points $M,C_2,P_1,P_2,P_3$ are concyclic.
  \label{th:circle_5}
\end{theorem}

\begin{figure}
    \centering
    \includegraphics[width=.7\textwidth]{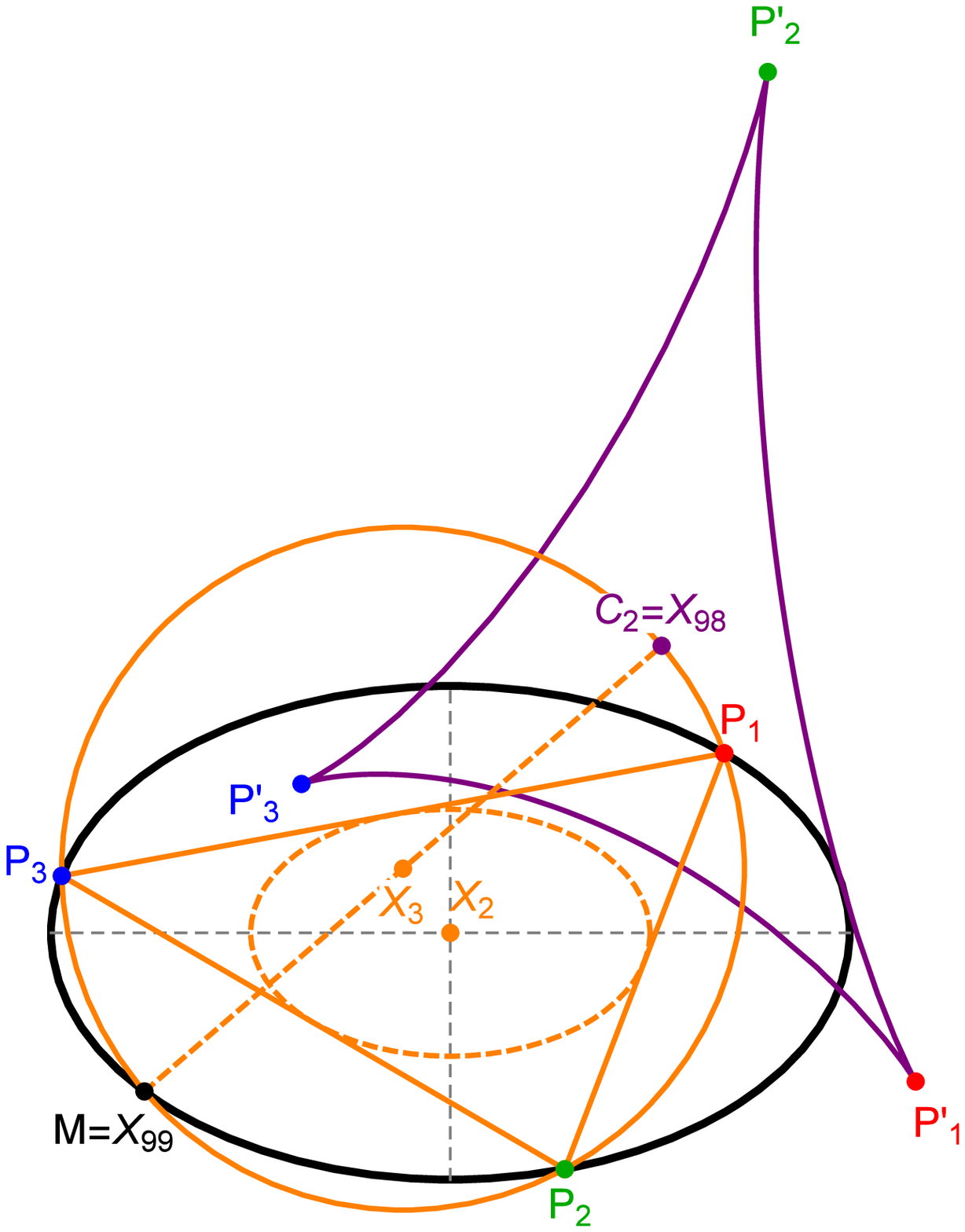}
    \caption{The cusp pre-images $P_i$ define a triangle $T$ (orange) whose area is invariant over $M$. Its barycenter $X_2$ is stationary at the center of $\E$, rendering the latter its Steiner Ellipse. Let $C_2$ denote the center of area of Steiner's Hat. The 5 points $M,C_2,P_1,P_2,P_3$ lie on a circle (orange), with center at $X_3$ (circumcenter of $T$). Over all $M$, the $T$ are a constant-area Poncelet family inscribed on $\E$ and tangent to a concentric, axis-aligned elliptic caustic (dashed orange), half the size of $\E$, i.e., the latter is the (stationary) Steiner Inellipse of the $T$. Note also that $M$ is the Steiner Point $X_{99}$ of $T$ since it is the intersection of its Circumcircle with the Steiner Ellipse. Furthermore, the Tarry Point $X_{98}$ of $T$ coincides with $C_2$, since it is the antipode of $M=X_{99}$ \cite{etc}. \textbf{Video:} \cite[PL\#02,\#05]{playlist2020-deltoid}.}
    \label{fig:preimg_tri}
\end{figure}

\begin{proof}
$\Delta_u(t)$ is singular at $t_1=-\frac{u}{3}$, $t_2=-\frac{u}{3}-\frac{2\pi}{3}$ and $t_3=-\frac{u}{3}-\frac{4\pi}{3}$. Let $P_i=[a\cos{t_i},b\sin{t_i}],i=1,2,3$. The circle $\K$ passing through these is given by:
 	\begin{equation}\label{eq:circulo}
 	\K(x,y)=x^2+y^2-\frac{c^2\cos u}{2a}x  
 	+\frac{c^2\sin u}{2b}y -\frac{1}{2}( a^2 +b^2)=0.
 	\end{equation}
 	Also, $\K(M)=\K(a\cos u,b\sin u)=0$.
 	
The center of $\K$ is $(M+C_2)/2$. It follows that $C_2\in \K$ and that $MC_2$ is a diameter of $\K$.
\end{proof}

\noindent In 1846, Jakob Steiner stated that given a point $M$ on an ellipse $\E$, there exist 3 other points on it such that the osculating circles at these points pass through $M$ \cite[page 317]{ostermann2012}. This property is also mentioned in \cite[page 97, Figure 3.26]{stachel2019-conics}.

It turns out the cusp pre-images are said special points! Referring to Figure~\ref{fig:osculating}:

\begin{proposition}
Each of the 3 circles $\K_i$ through $M,P_i,P_i'$, $i=1,2,3$, osculate $\E$ at $P_i$.
\label{prop:osculate}
\end{proposition}

\begin{figure}
    \centering
    \includegraphics[width=.6\textwidth]{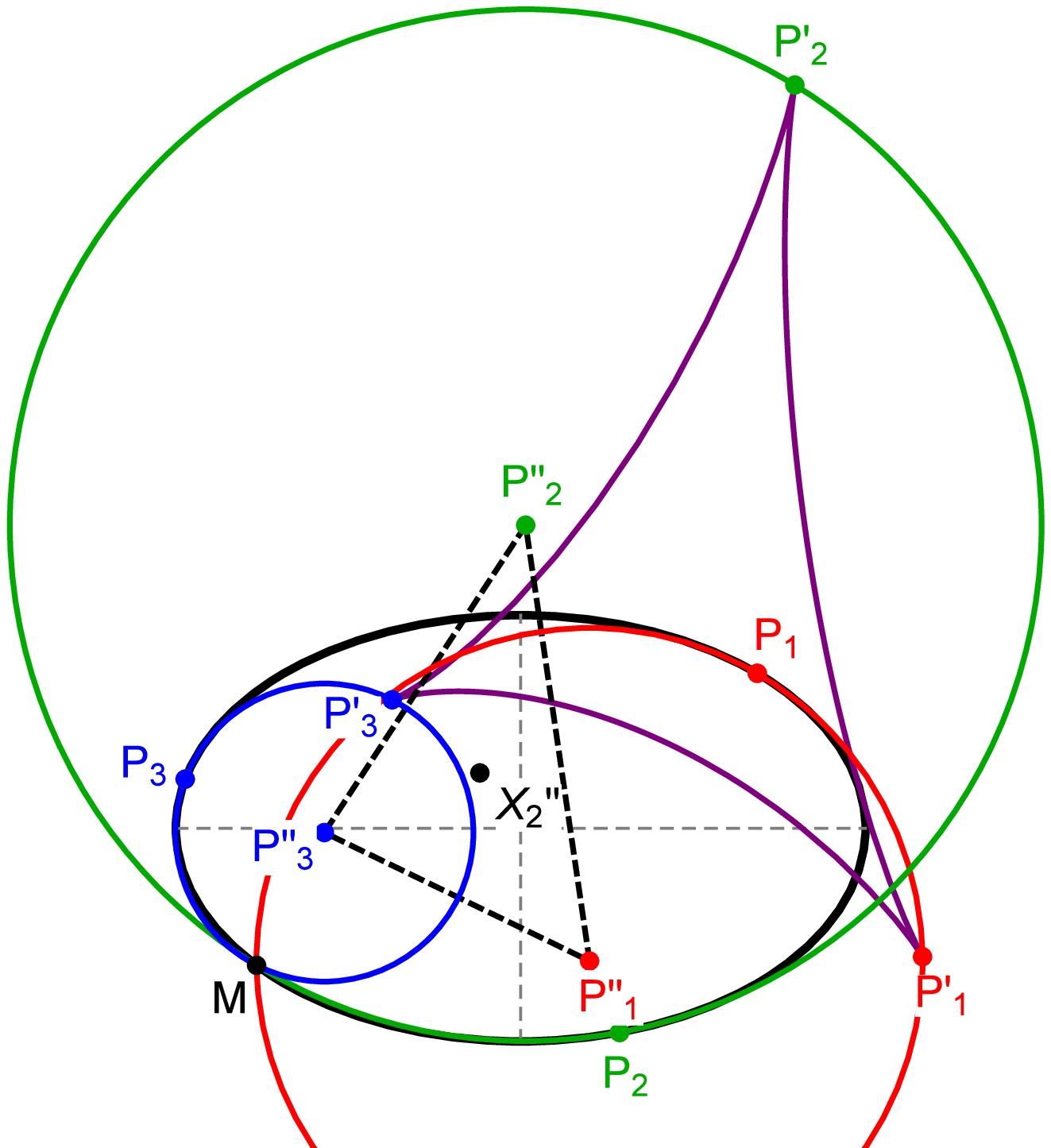}
    \caption{The circles passing through a cusp $P_i'$, its pre-image $P_i$, and $M$ osculate $\E$ at the $P_i$. The centers $P_i''$ of said circles define a triangle $T''$ (dashed black) whose area is constant for all $M$. $X_2''$ denotes its (moving) barycenter. \textbf{Video:} \cite[PL\#03,\#05]{playlist2020-deltoid}.}
    \label{fig:osculating}
\end{figure}

\begin{proof} The circle $\K_1$ passing through $M$, $P_1$ and $P_1'$ is given by

\[ \aligned \K_1(x,y)=& 2\,ab ({x}^{2}+{y}^{2})-4bc^2  \cos^3
 \left( \frac{u}{3 }  \right)    \; x-4ac^2    \sin^3 \left( \frac{u}{3 }\right)   y\\
+&ab \left( 3c^2\,\cos\left(  \frac{2u}{3 }\right)   -{a}^{2}-{b}^{2} \right) =0.
\endaligned \]

\noindent Recall a circle osculates an ellipse if its center lies on the evolute of said ellipse, given by \cite{stachel2019-conics}: 

\begin{equation}
\E^{*}(t)= \left[ \frac {  c^2\; \cos^3 t     }{a},-
\frac {{c}^{2}   \sin^3 t   }{b}\right]
\end{equation}

It is straightforward to verify that the center of $\K_1$ is $P_1''=\E^{*}(-\frac{u}{3})$.
A similar analysis can be made for $\K_2$ and $\K_3$.
\end{proof}

Since the area of $\E^*$ is $ A(\E^{*})=\frac{3 \pi c^4}{8{a}{b}}$, and the area of $\Delta$ is given in Equation~\ref{eqn:adelta}:

\begin{remark}
The area ratio of $\Delta$ and the interior of $\E^*$ is equal to $4/3$. \textcolor{red}{\smiley{}}
\end{remark}

\subsection{Why is \texorpdfstring{$\Delta$}{Δ} affine to Steiner's Curve}

Up to projective transformations, there is only one irreducible curve of degree 4 with 3 cusps. In a projective coordinate frame $(x_0:x_1:x_2)$ with the cusps as base points $(1:0:0)$, $(0:1:0)$ and $(0:0:1)$ and the common point of the cusps' tangents as unit point (1:1:1), the quartic has the equation

\[ x_0^2 x_1^2 + x_0^2 x_2^2 + x_1^2 x_2^2 - 2 x_0 x_1 x_2 (x_0+x_1+x_2) = 0 \]

At Steiner's three-cusped curve, the cusps form a regular triangle with the tangents passing through the center. Hence, whenever a three-cusped quartic has the meeting point of the cusps' tangents at the center of gravity of the cusps, it is affine to Steiner's curve, since there is an affine transformation sending the four points into a regular triangle and its center.

\section{The Cusp Triangle}
\label{sec:cusp_tri}
Recall $T'=P_1'P_2'P_3'$ is the triangle defined by the 3 cusps of $\Delta$.

\begin{proposition}
The area $A'$ of the cusp triangle $T'$ is invariant over $M$ and is given by:
 \[  A'=
\frac{27 \sqrt{3}}{16}\; \frac{ c^4 }{   a b}
\label{invAreaT'}
\]
\end{proposition}
\begin{proof}


The determinant of the Jacobian of the affine transformation in Theorem~\ref{th:affine} is $|J|=\frac{c^4}{4ab}$. Therefore, the area of $T'$ is simply $|J|A_e$, where $A_e$ is the area of an equilateral triangle inscribed in a circle of radius $3$ 
with side $3\sqrt{3}$.
\end{proof}

\noindent Referring to Figure~\ref{fig:deltoid_c2}:

\begin{proposition}
The barycenter $X_2'$ of $T'$ coincides with the center of area $C_2$ of $\Delta$.
\label{prop:c2-delta}
\end{proposition}
\begin{proof}
Direct calculations yield $X_2'=C_2.$
\end{proof}

\noindent Referring to Figure~\ref{fig:cusp_steiner}:

\begin{proposition}
The Steiner Ellipse $\E'$ of $T'$ has constant area and is a scaled version of $\E$ rotated $90^\circ$ about $O$.
\end{proposition}

\begin{proof} $\E'$ passes through the vertices of $T'$ and is centered on $C_2=X_2'$. Direct calculations yield the following implicit equation for it:
	
\[ \aligned   {a}^{2}{x}^{2}+  {b}^{2}{y}^{2} + \left( {a}^{2}+{b}^{2} \right) \left(  a\cos
	u   x+  b \sin u y\right)- \left( {a}^{2}-2\,{
		b}^{2} \right)  \left( 2\,{a}^{2}-{b}^{2} \right)=0
	\endaligned
\]
Its semi-axes are $b'=\frac{3c^2}{2a}$ and $a'=\frac{3c^2}{2b}$. Therefore $\E'$ is similar to a $90^\circ$-rotated copy of $\E$.
\end{proof}
\begin{remark}
This proves that $T'$ can never be regular and $\Delta $ has never a three-fold symmetry.\end{remark}

\begin{corollary}
The ratio of area of $\E'$ and $\E$ is given by $\frac{9}{4}\frac{c^4}{a^2 b^2}$, and at $a/b=(1+\sqrt{10})/3$, the two ellipses are congruent.
\end{corollary}

\begin{figure}
    \centering
    \includegraphics[width=\textwidth]{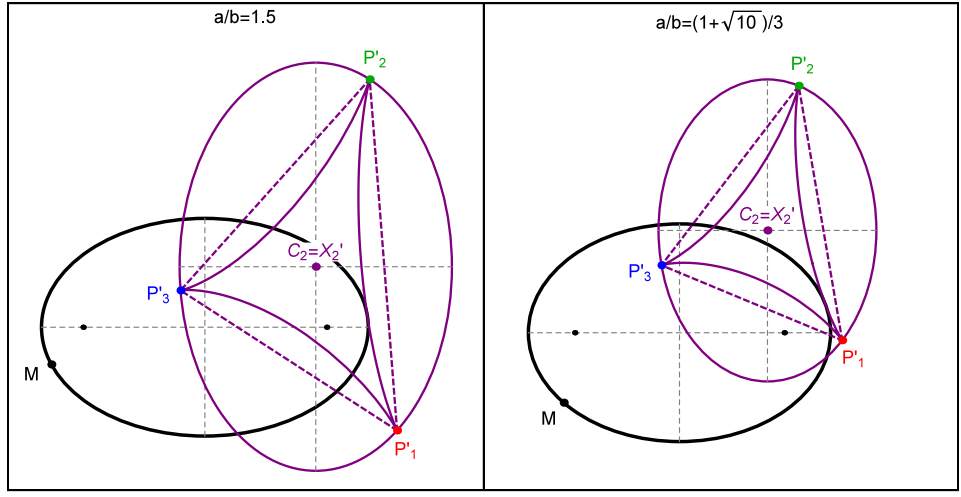}
    \caption{\textbf{Left}: The Steiner Ellipse $\E'$ of triangle $T'$ defined by the $P_i'$ is a scaled-up and $90^\circ$-rotated copy of $\E$. \textbf{Right:} At $a/b=(1+\sqrt{10})/3{\simeq}1.38743$, $\E'$ and $\E$ have the same area.}
    \label{fig:cusp_steiner}
\end{figure}

\begin{proposition}
 The Steiner Point $X_{99}'$ of $T'$ is given by: 
 
  \[ X_{99}'(u) = \left[ \frac { \left( {a}^{2}-2\,{b}^{2} \right)  } {a}\cos u
,- \frac { \left( 2\,{a}^{2}-{b}^{2} \right) }{
b} \sin u  \right]
\]
\end{proposition}

\begin{proof}
By definition $X_{99}$ is the intersection of the circumcircle of $T$ ($\K$) with  the Steiner ellipse.
 The Circumcircle $\K'$ of the triangle 
     $T'=\{P_1',P_2', P_3'\}$ is given by:
     \begin{align*}
\label{eq:cc}
    \K'(x,y)=&8\,{a}^{2}{b}^{2} \left( {x}^{2}+{y}^{2} \right) \\
     +&2\,a\cos u \left( 3\,{a}^{4}-2\,{a}^{2}{b}^{2}+7\,{b}^{4} \right) x 
     + 2\,b\sin u \left( 7\,{a}^{4}-2\,{a}^{2}{b}^{2}+3\,{b}^{4}
 \right) y\\
 -& \left( {a}^{2}+{b}^{2} \right)  \left(  c^2  \left( {a}^{2}+{b}^{
2} \right) \cos   2\,u
  +5\,{a}^{4}-14\,{a}^{2}{b}^{2}+5\,{b}^{4} \right) =0.
     \end{align*}

With the above, straightforward calculations lead to the coordinates of $X_{99}'$.
\end{proof}

\section{The Triangle of Cusp Pre-Images}
\label{sec:preimg_tri}
Recall $T={P_1}{P_2}{P_3}$ is the triangle defined by pre-images on $\E$ to each cusp of $\Delta$. 

\begin{proposition}
The barycenter $X_2$ of $T$ is stationary at $O$, i.e., $\E$ is is Steiner Ellipse \cite{mw}.
\end{proposition}
\begin{proof} The triangle $T$ is an affine image of an equilateral triangle with center at $0$  and $P_i=\E(t_i)=\E(-\frac{u}{3}-(i-1)\frac{2\pi}{3})$. So the result follows. 
\end{proof}

\begin{remark}
$M$ is the Steiner Point $X_{99}$ of $T$.
\end{remark}

\begin{proposition}
Over all $M$, the $T$ are an $N=3$ Poncelet family with external conic $\E$ with the Steiner Inellipse of $T$ as its caustic \cite{mw}. Futhermore the area of these triangles is invariant and equal to $\frac{3\sqrt{3}ab}{4}$.
\end{proposition}

\begin{proof}
The pair of concentric circles of radius $1$ and $1/2$ is associated with a Poncelet 1d family of equilaterals. The image of this family by the map $(x,y)\rightarrow(ax,by)$ produces the original pair of ellipses, with the stated area. Alternatively, the ratio of areas of a triangle to its Steiner Ellipse is known to be $3\sqrt{3}/(4\pi)$ \cite[Steiner Circumellipse]{mw} which yields the area result. 
\end{proof}

\section{Locus of the Cusps}
\label{sec:cusp_locus}
We analyze the motion of the cusps $P_i'$ of Steiner's Hat $\Delta$ with respect to continuous revolutions of $M$ on $\E$. Referring to Figure~\ref{fig:cusp-loci}:

\begin{remark}
The locus $\mathcal{C}(u)$ of the cusps of $\Delta$ is parametrized by:

\begin{align} \mathcal{C}(u): \frac{3c^2}{2} \left[ \frac{1}{a}\cos\frac{u}{3}  ,\frac{1}{b}\sin\frac{u}{3}\right]-\frac{a^2+b^2}{2} \left[ \frac{1}{a}\cos u ,\frac{1}{b}\sin u\right]
\end{align}

This is a curve of degree 6, with the following implicit equation:

{\scriptsize
\begin{align*}
-&4\,{a}^{6}{x}^{6}-4\,{b}^{6}{y}^{6}-12\,{a}^{2}{x}^{2}{b}^{2}{y}^{2}
 \left( {a}^{2}{x}^{2}+{b}^{2}{y}^{2} \right)\\ +&12\,{a}^{4} \left( {a}^
{4}-{a}^{2}{b}^{2}+{b}^{4} \right) {x}^{4}+12\,{b}^{4} \left( {a}^{4}-
{a}^{2}{b}^{2}+{b}^{4} \right) {y}^{4}+24\,{a}^{2}{b}^{2} \left( {a}^{
4}-{a}^{2}{b}^{2}+{b}^{4} \right) {x}^{2}{y}^{2}\\
-&3\,{a}^{2} \left( 2\,
{a}^{2}-{b}^{2} \right)  \left( {a}^{2}+{b}^{2} \right)  \left( 2\,{a}
^{4}-5\,{a}^{2}{b}^{2}+5\,{b}^{4} \right) {x}^{2}\\
+&3\,{b}^{2} \left( {a
}^{2}-2\,{b}^{2} \right)  \left( {a}^{2}+{b}^{2} \right)  \left( 5\,{a
}^{4}-5\,{a}^{2}{b}^{2}+2\,{b}^{4} \right) {y}^{2}\\
+& \left( 2\,{a}^{2}-
{b}^{2} \right) ^{2} \left( {a}^{2}-2\,{b}^{2} \right) ^{2} \left( {a}
^{2}+{b}^{2} \right) ^{2}
 =0
\end{align*}
}

    %
\end{remark}

\begin{proposition}
It can be shown that over one revolution of $M$ about $\E$, $C_2$ will cross the ellipse on four locations $W_j,j=1,\cdots,4$ given by:

\[ W_j= \frac{1}{2 \sqrt{a^2+b^2}}\left(\pm a \sqrt{a^2+3 b^2} ,\pm b \sqrt{3 a^2+b^2} \right) \]

At each such crossing, $C_2$ coincides with one of the pre-images.
\end{proposition}

\begin{proof}
From the coordinates of $C_2$ given in equation~\eqref{eqn:cbar} in terms of the parameter $u$, one can derive an equation that is a necessary and sufficient condition for $C_2\in \E$ to happen, by substituting those coordinates in the ellipse equation $x^2/a^2+y^2/b^2-1=0$. Solving for $\sin u$ and substituting back in the coordinates of $C_2$, one easily gets the four solutions $W_1,W_2,W_3,W_4$.

Now, assume that $C_2\in \E$. The points $M,P_1,P_2,P_3,C_2$ must all be in both the ellipse $\E$ and the circumcircle $\K$ of $P_1P_2P_3$. Since the two conics have at most 4 intersections (counting multiplicity), 2 of those 5 points must coincide. It is easy to verify from the previously-computed coordinates that $M$ can only coincide with the preimages $P_1,P_2,P_3$ at the vertices of $\E$. In such cases, owing to the symmetry of the geometry about either the x- or y-axis, the circle $\K$ must be tangent to $\E$ at $M$. Thus, that intersection will count with multiplicity (of at least) 2, so another pair of those 5 points must also coincide. Since $P_1,P_2,P_3$ must all be distinct, $C_2$ will coincide with one of the preimages. However, this can never happen, since if $M$ is on one of the vertices of $\E$, $C_2$ won't be in $\E$.
In any other case, since $P_1,P_2,P_3$ must be distinct and $C_2$ is diametrically opposed to $M$ in $\K$, $C_2$ must coincide with one of the preimages.
\end{proof}

\begin{remark}
$C_2$ will visit each of the preimages cyclically. Moreover, upon 3 revolutions (with 12 crossings in total), each $P_i$ will have been visited four times and the process repeats. 
\end{remark}

\begin{figure}[H]
    \centering
    \includegraphics[width=\textwidth]{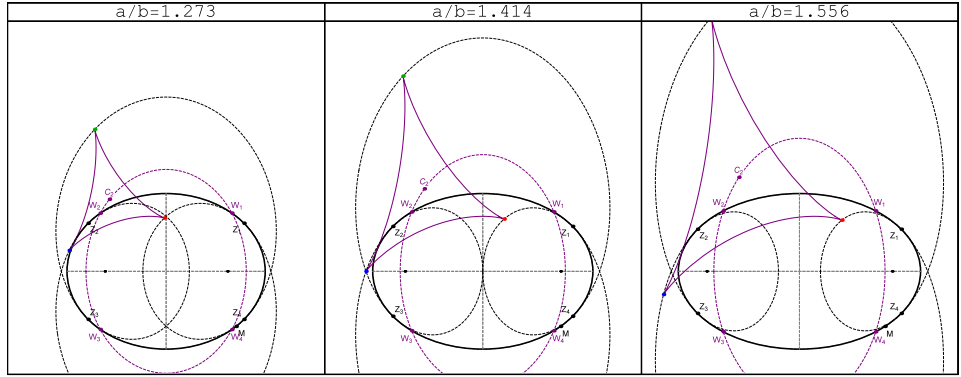}
    \caption{The loci of the cusps of $\Delta$ (dashed line) is a degree-6 curve with 2 internal lobes with either 2, 3, or 4 self-intersections. From left to right, $a/b=\{1.27,\sqrt{2},1.56\}$. Note that at $a/b=\sqrt{2}$ the two lobes touch, i.e., the cusps pass through the center of $\E$. Also shown is the elliptic locus of $C_2$ (purple). Points $Z_i$ (resp. $W_i$) mark off the intersection of the locus of the cusps (resp. of $C_2$) with $\E$. These never coincide \textbf{Video:} \cite[PL\#04]{playlist2020-deltoid}.} 
    \label{fig:cusp-loci}
\end{figure}

\subsection{Tangencies and Intersections of the Deltoid with the Ellipse}
\label{sec:deltoid_tangs}
\begin{definition}[Apollonius Hyperbola]
Let $M$ be a point on an ellipse $\E$ with semi-axes $a,b$. Consider a hyperbola $\H$, known as the Apollonius Hyperbola of $M$ \cite{hartmann2010-apollonius}:

\[ \H: \langle (x,y)-M, (y/b^2,-x/a^2 )\rangle=0. \] 

\end{definition}

Notice that for $P$ on $\E$, only the points for which the normal at $P$ points to $M$ will lie on $\H$. See also \cite[page 403]{stachel2019-conics}.

Additionally, $\H$ passes through $M$ and $O$, and its asymptotes are parallel to the axes of $\E$.  
 
\begin{proposition}\label{prop:3tan}
$\Delta$ is tangent to $\E$ at $\E\cap\H$, at 1, 2 or 3 points $Q_i$ depending on whether $M$ is exterior or interior to the evolute $\E^*$. 
\end{proposition}
\begin{proof}
$\Delta$ is tangent to $\E$ at some $Q_i$ if the normal of $\E$ at $Q_i$ points to $M=(M_x,M_y)$, i.e., when  $\H$ intersects with $\E$. It can be shown that their $x$ coordinate is given by the real roots 
of:

\begin{equation}
\mathcal{Q}(x) =  c^4 x^3- c^2M_x(a^2+b^2)x^2-a^4(a^2-2b^2)x+a^6M_x =0
\label{eqn:qi-cubic}
\end{equation}

The discriminant of the above is:
\[ -4c^4a^6(a^2-M_x^2)[(a^2-b^2) (a^2+b^2)^3M_x^2+a^4(a^2-2b^2)^3].\]

Let $\pm x^*$ denote the solutions to $(a^2-b^2) (a^2+b^2)^3M_x^2+a^4(a^2-2b^2)^3=0$. Assuming $a>b$, Equation~\ref{eqn:qi-cubic} has three real solutions when $|x|<x^*.$ 
The intersections of the evolute $\E^{*}$ with the ellipse $\E$ are given by the four points $(\pm x^*,\pm y^*)$, where:

\begin{equation}
x^*= \dfrac { a^2 \sqrt {  {a}^{4}-{b}^{4}    } \left( {a}^{2}-2\,{b}^{2} \right)^{\frac{3}{2}} }{ \left( {a}^{4}-{b}
^{4} \right)  \left( {a}^{2}+{b}^{2} \right) },\;\;\;y^*=  \dfrac {b^2 \sqrt {{a}^{4}-{b}^{
4}}\left( 2\,{a}^{2}-{b}^{2} \right) ^{\frac{3}{2}}}{ \left( {a}^{4}-{b}
^{4} \right)  \left( {a}^{2}+{b}^{2} \right) }
\label{eqn:ev-inter}
\end{equation}
For $M \in \E\cap\E^\ast$ two coinciding roots result in a 4-point contact between $\Delta$ and the ellipse.
\end{proof}

\begin{figure}
    \centering
    \includegraphics[width=\textwidth]{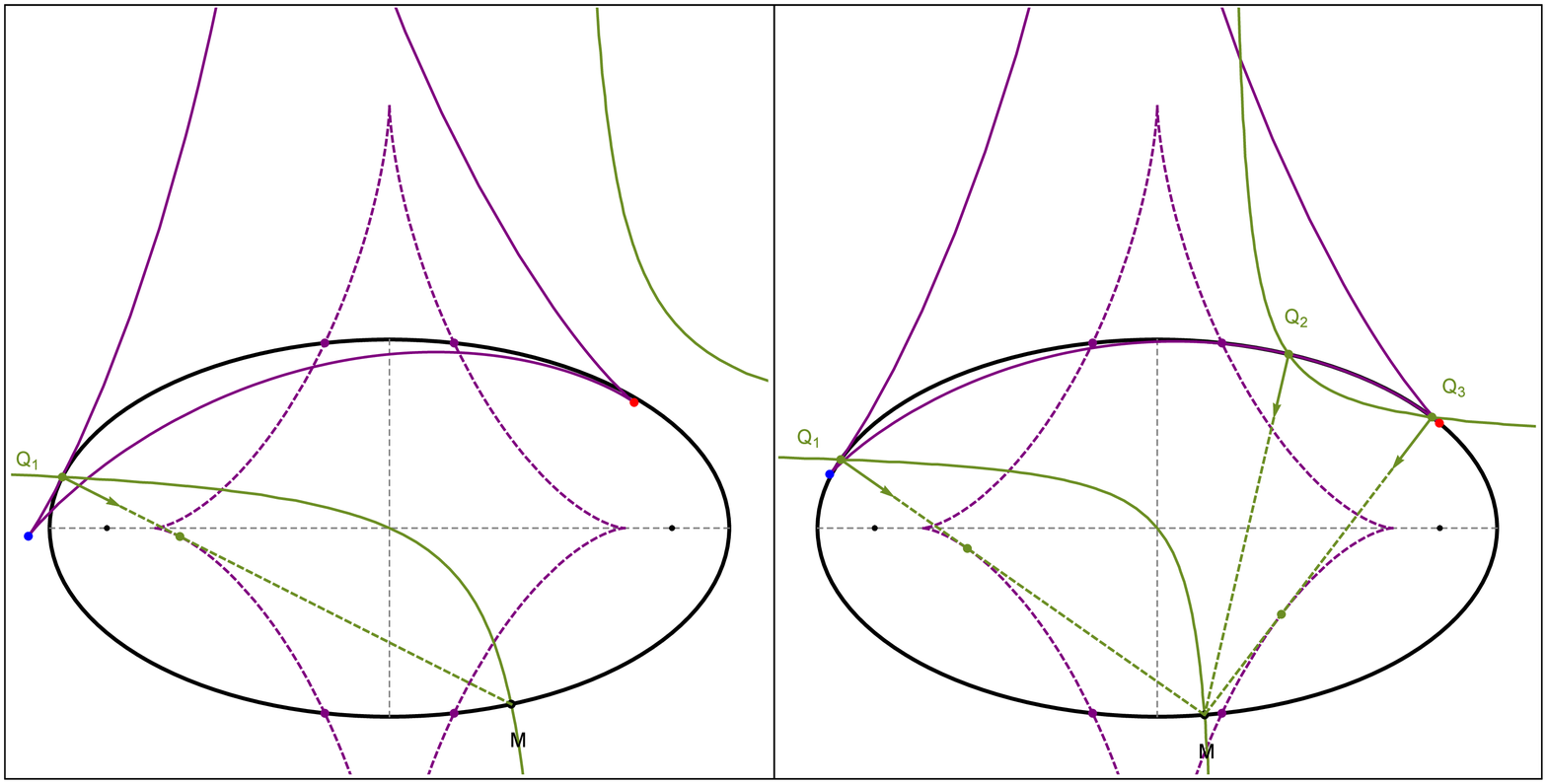}
    \caption{Steiner's Hat $\Delta$ (purple, top cusp not shown) is tangent to $\E$ at the intersections $Q_i$ of the Apollonius Hyperbola $\H$ (olive green) with $\E$, excluding $M$. Notice $\H$ passes through the center of $\E$. \textbf{Left:} When $M$ is exterior to the evolute $\E^*$ (dashed purple), only one tangent $Q_1$ is present. \textbf{Right:} When $M$ is interior to $\E^*$, three tangent points $Q_i,i=1,2,3$ arise. The intersections of $\E^*$ are given in Equation~\ref{eqn:ev-inter}. Note: the area ratio of $\Delta$-to-$\E^*$ is always $4/3$.}
    \label{fig:deltoid-tangs}
\end{figure}

Let $M=(M_x,M_y)$ be a point on $\E$ and $\mathcal{J}(x)$ denote the following cubic polynomial:

\begin{equation}
\label{eq:j-inter}
\mathcal{J}(x)=(a^2+b^2)^2x^2-2M_x c^2(a^2+b^2)x-4a^4b^2+M_x^2(a^2+b^2)^2\\
\end{equation}
\begin{proposition}
$\Delta$ intersects $\E$ at $\mathcal{Q}(x)\mathcal{J}(x)=0$, in at least $3$ and up to $5$ locations locations, where $\mathcal{Q}$ is as in Equation~\ref{eqn:qi-cubic}.
\end{proposition}

\begin{proof}
As before, $M=(M_x,M_y)=(a\cos{u},b\sin{u})\in \E$ and $ P=(x,y)=(a\cos{t},b\sin{t})$. The intersection $\Delta_u(t)$ with  $\E$ is obtained by setting $\E(\Delta_u(t))=0$. Using Equation~\ref{eq:deltoide}, obtain the following system:

	\begin{align*}
F(x,y)=&	{b}^{2} \left( {a}^{2}-2\,M_x^2 \right)  \left( {a}^{2}+{b}^{
		2} \right) c^{4}{x}^{4}+2\,{a}^{2}M_x\,{M_y}\,
	\left( {a}^{2}+{b}^{2} \right) c^{4}{x}^{3}y\\
	+&2\,{a}^{2}{b}^{
		2}{ M_x}\,c^{2} \left( {a}^{4}+{b}^{4} \right) {x}^{3}-2\,{
		a}^{4}M_y\,c^{2} \left( {a}^{4}+{b}^{4} \right) {x}^{2}y\\
	+& \left[  -{a}^{4}{b}^{2} \left( {a}^{2}+{b}^{2} \right)  \left( 3\,{a}
	^{4}-4\,{a}^{2}{b}^{2}+2\,{b}^{4} \right) +{a}^{2}{b}^{2} c^2 M_x^2
	 \left( 3\,{a}^{2}-{b}^{2} \right)  \left( {a}^{2}+{b}^
	{2} \right)  \right] {x}^{2}\\
	-&2\,{a}^{6}M_x\,{M_y}\,{c}^{2} \left( {a}^{2}+{b}^{2} \right) xy-2\,{b}^{2}{  M_x}\,{a}^{6}
	\left( {a}^{4}-{a}^{2}{b}^{2}+{b}^{4} \right) x\\
	+&2\,{a}^{12} {  M_y}y+
	{a}^{8}{b}^{2} \left( 2\,{a}^{4}-({a}^{2}+{b}^{2}) M_x^{2} \right)=0 \\
\E(x,y)=& \frac{x^2}{a^2}+\frac{y^2}{b^2}-1=0
	\end{align*}

By Bézout's theorem, the system $\E(x,y)=F(x,y)=0$ has eight solutions, with algebraic multiplicities taken into account. $\Delta$ has three points of tangency with ellipse, some of which may be complex, which by Proposition~\ref{prop:3tan} are given by the zeros of $\mathcal{Q}(x)$.
Eliminating $y$ by computing the resultant we obtain
an Equation $G(x)=0$ of degree 8 over $x$. Manipulation with a Computer Algebra System yields a compact representation for $G(x)$:
\[ G(x)= \mathcal{Q}(x)^2 \mathcal{J}(x)\]
\noindent with $\mathcal{J}$ as in Equation~\ref{eq:j-inter}. If $|M_x|\leq a$,	the solutions of $\mathcal{J}(x)=0$ are real and given by
$J_x=[c^2M_x\pm 2 a b\sqrt{a^2-M_x^2}]/(  a^2+b^2)$ and $|J_x|\leq a$.

\end{proof}

\noindent Referring to Figure~\ref{fig:cusp-loci}:

\begin{proposition}
\label{prop:coincidence}
 When a cusp $P_i'$ crosses the boundary of $\E$, it coincides with its pre-image $P_i$ at $Z_i=\frac{1}{\sqrt{a^2+b^2}}(\pm a^2,\pm b^2)$. 
\end{proposition}

\begin{proof}  
Assume $P_i'$ is on $\E$. Since $P_i'$ is on the circle $\K_i$ defined by $M$, $P_i$ and $P_i'$ which osculates $\E$ at $P_i$, this circle intersects $\E$ at $P_i$ with order of contact 3 or 4.
By construction, we have  $M P_i\perp P_i P_i'$, so $M$ and $P_i'$ are diametrically opposite in $\K_i$. Thus, $M$ and $P_i'$ must be distinct. Since two conics have at most 4 intersections (counting multiplicities), we either have $P_i'=P_i$  or $P_i=M$. The second case will only happen when $M$ is on one of the four vertices of the ellipse $\E$, in which case the osculating circle $\K_i$ has order of contact 4, so $P_i'$ could not also be in the ellipse in the first place. Thus, $P_i'=P_i$ as we wanted.

Substituting the parameterization of $P_i$ in the equation of $\E$ , we explicitly find the four points $Z_i$ at which $P_i'$ can intersect the ellipse $\E$. 
\end{proof}


\section{A Triad of Osculating Circles}
\label{sec:osc_circs}

Recall $\K_i$ are the circles which osculate $\E$ at the pre-images $P_i$, see Figure~\ref{fig:osculating}. Define a triangle $T''$ by the centers $P_i''$ of the $\K_i$. These are given explicit coordinates in Appendix~\ref{app:cusps}. Referring to Figure~\ref{fig:osc_concur}:

\begin{proposition}
Triangles $T'$ and $T''$ are homothetic at ratio $2:1$, with $M$ as the homothety center.
\label{prop:persM}
\end{proposition}

\begin{proof}
From the construction of $\Delta$, for each $i=1,2,3$ we have $M P_i\perp P_i P_i'$, that is, $\angle M P_i P_i'=90^{\circ}$. Hence, $M P_i'$ is a diameter of the osculating circle $\K_i$ that goes through $M$, $P_i$, and $P_i'$ as proved in Proposition ~\ref{prop:osculate}. Thus, the center $P_i''$ of $\K_i$ is the midpoint of $M P_i'$ and therefore $P_i'$ is the image of $P_i''$ under a homothety of center $M$ and ratio 2.
\end{proof}

\begin{corollary}
The area $A''$ of $T''$ is invariant over all $M$ and is $1/4$ that of $T'$.
\end{corollary}

\begin{proof}
This follows from the homothety, and the fact that the area of $T'$ is invariant from Proposition~\ref{invAreaT'}.
\end{proof}

\begin{figure}
    \centering
    \includegraphics[width=.8\textwidth]{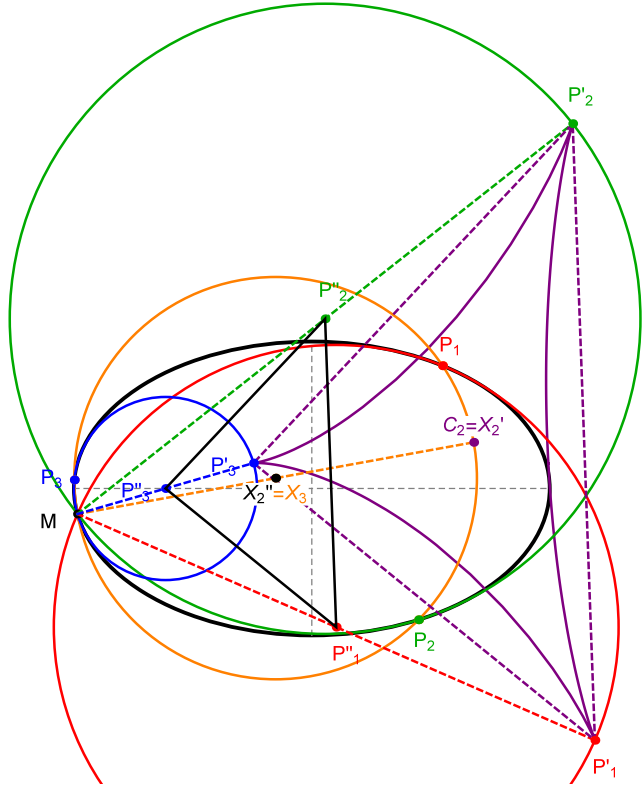}
    \caption{Lines connecting each cusp $P_i'$ to the center $P_i''$ of the circle which osculates $\E$ at the pre-image $P_i$ concur at $M$. Note these lines are diameters of said circles. Therefore $M$ is the perspector of $T'$ and $T''$, i.e., the ratio of their areas is 4. This perspectivity implies $C_2,X_2'',M$ are collinear. Surprisingly, the $X_2''$ coincides with the circumcenter $X_3$ of the pre-image triangle $T$ (not drawn).}
    \label{fig:osc_concur}
\end{figure}

\begin{proposition} Each (extended) side of $T'$ passes through an intersection of two osculating circles.
Moreover, those sides are perpendicular to the radical axis of said circle pairs.
\end{proposition}

\begin{proof}
It suffices to prove it for one of the sides of $T'$ and the others are analogous. Let $M_1$ be the intersection of $\K_2$ and $\K_3$ different than $M$ and let $M_{1/2}$ be the midpoint of $M$ and $M_1$. Since $M M_1$ is the radical axis of $\K_2$ and $\K_3$, the lines $P_2''P_3''$ and $M M_1$ are perpendicular and their intersection is $M_{1/2}$. Applying an homothety of center $M$ and ratio 2, we get that the lines $P_2' P_3'$ and $M M_1$ (the radical axis) are perpendicular and their intersection is $M_1$, as desired.
\end{proof}

\begin{corollary}
The Steiner ellipse $\E''$ of triangle  $T''$ is similar to $\E'$. In fact, $\E''=\frac{1}{2}\E'.$
\end{corollary}

\begin{proof}
This follows from the homothety of $T'$ and $T''$.
\end{proof}

\section{Relations between \texorpdfstring{$T,T',T''$}{T,T',T''}}
\label{sec:triangles}
As before we identify Triangle centers as $X_k$ after Kimberling's Encyclopedia \cite{etc}.

\begin{proposition}\label{prop:X3X2pp}
The circumcenter $X_3$ of $T$ coincides with the barycenter $X_2''$ of $T''$.
\end{proposition}
\begin{proof} Follows from direct calculations using the coordinate expressions of $P_i$ and $P_{i}''$. In fact,
\[ X_2''=X_3=  \frac{c^2}4 \left[ \frac{\cos u}a, - \frac{\sin u}b \right].
\]
\end{proof}
\begin{corollary} The homothety with center $M$ and factor $2$ sends $X_2''$ to $X_2'=C_2$.

\end{corollary}

\begin{proposition}
The lines joining a cusp $P_i'$ to its preimage $P_i$ concur at $\Delta$'s center of area $C_2$.
\end{proposition}
\begin{proof}
From Theorem~\ref{th:circle_5}, points $M$ and $C_2$ both lie on the circumcircle of $T$ and form a diameter of this circle. Thus, for each $i=1,2,3$, we have $\angle M P_i C_2=90^{\circ}$. By construction, $\angle M P_i P_i'=90^{\circ}$, so $P_i$, $P_i'$, and $C_2$ are collinear as desired.
\end{proof}

Referring to Figure~\ref{fig:persp-t-tp}(left):
\begin{corollary}  
 $C_2=X_2'$ is the perspector of $T'$ and $T$.
\label{perspTTp}
 
\end{corollary}

\begin{lemma}
Given a triangle $\T$, and its Steiner Ellipse $\Sigma$, the normals at each vertex pass through the Orthocenter of $\T$, i.e., they are the altitudes.
\label{lem:steiner_x4}
\end{lemma}

\begin{proof}
This stems from the fact that the tangent to $\Sigma$ at a vertex of $\T$ is parallel to opposide side of $\T$ \cite[Steiner Circumellipse]{mw}.
\end{proof}

Referring to Figure~\ref{fig:persp-t-tp}(right):

\begin{proposition} The orthocenter 
$X_4$ is the perspector of $T$ and $T'' $. Equivalently, a line connecting a vertex of $T$ to the respective vertex of $T''$ is perpendicular to the opposite side of $T$.
\label{prop:perspTTpp}
\end{proposition}

\begin{proof}
Since $T$ has fixed $X_2$, $\E$ is its Steiner Ellipse. The normals to the latter at $P_i$
pass through centers $P_i''$ since these are osculating circles. So by  Lemma~\ref{lem:steiner_x4} the proof follows.
\end{proof}

\begin{definition} According to J.\ Steiner \cite[p.~55]{Gallatly}, two triangles $ABC$ and $DEF$ are said to be \textit{orthologic} if the perpendiculars from $A$ to $EF$, from $B$ to $DF$, and from $C$ to $DE$ are concurrent. Furthermore, if this holds, then the perpendiculars from $D$ to $BC$, from $E$ to $AC$, and from $F$ to $AB$ are also concurrent. Those two points of concurrence are called the \textit{centers of orthology} of the two triangles \cite{patrascu2020}.
\end{definition}

Note that orthology is symmetric but not transitive \cite[p. 37]{patrascu2020}, see Figure~\ref{fig:non-transitive} for a non-transitive example involving a reference, pedal, and antipedal triangles.

\begin{figure}
    \centering
    \includegraphics[width=.8\textwidth]{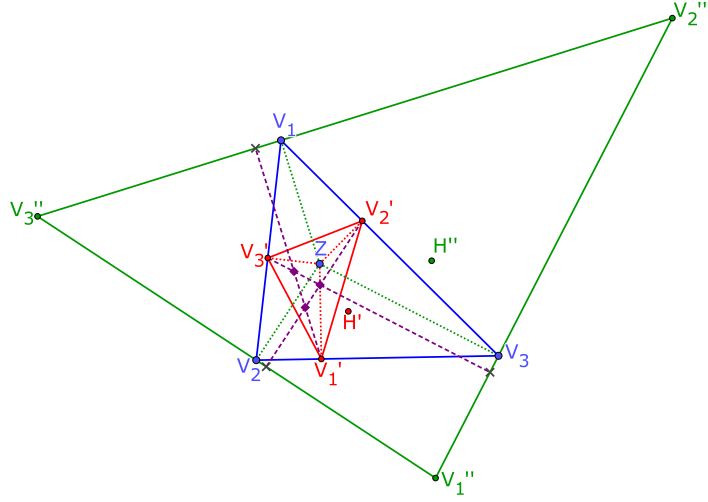}
    \caption{Consider a reference triangle $\T$ (blue), and its pedal $\T'$ (red) and antipedal $\T''$ (green) triangles with respect to some point $Z$. Construction lines for both pedal and antipedal  (dashed red, dashed green) imply that $Z$ is an orthology center simultaneousy for both $\T,\T'$ and $\T,\T''$, i.e. these pairs are orthologic. Also shown are $H'$ and $H''$, the 2nd orthology centers of said pairs (construction lines omitted). Non-transitivity arises from the fact that perpendiculars dropped from the vertices of $\T'$ to the sides of $\T''$ (dashed purple, feet are marked X) are non-concurrent (purple diamonds mark the three disjoint intersections), i.e., $\T',\T''$ are not orthologic.} 
    \label{fig:non-transitive}
\end{figure}

\begin{lemma}
Let $-P_i$ denote the reflection of $P_i$ about $O=X_2$ for $i=1,2,3$. Then the line from $M$ to $-P_1$ is perpendicular to the line $P_2''P_3''$, and analogously for $-P_2$, and $-P_3$.
\label{lem:perp-pi}
\end{lemma}

\begin{proof} This follows directly from the coordinate
expressions for points $M$, $P_i=\E(t_i)$ and $P_i''=\E^{*}(t_i)$.
It follows that $\langle M+P_1,P_2''-P_3''\rangle=0.$
\end{proof}

Referring to Figure~\ref{fig:persp-t-tp}(right):

\begin{theorem}
\label{th:orthologic}
Triangles $T$ and $T'$ are orthologic and their centers of orthology are the reflections  $X_{671}$ of $M$  on $X_2$ and on $X_4$.
\end{theorem}

\begin{figure}
\centering
\begin{minipage}[b]{0.45\linewidth}
\frame{\includegraphics[height=.25\textheight,trim=10 0 40 30 ,clip]{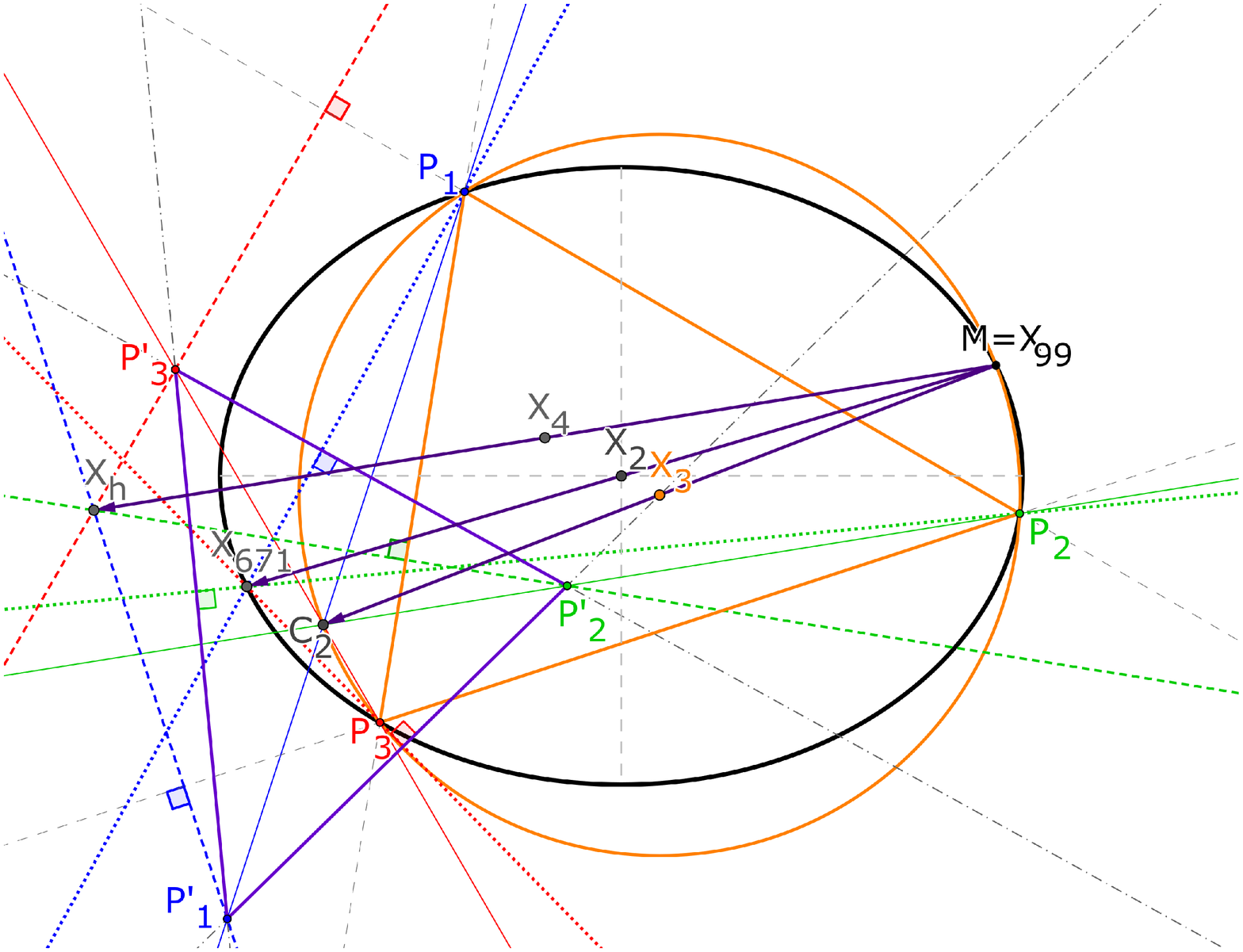}}
\end{minipage}
\quad
\begin{minipage}[b]{0.45\linewidth}
\frame{\includegraphics[height=.25\textheight,trim=75 0 60 30 ,clip]{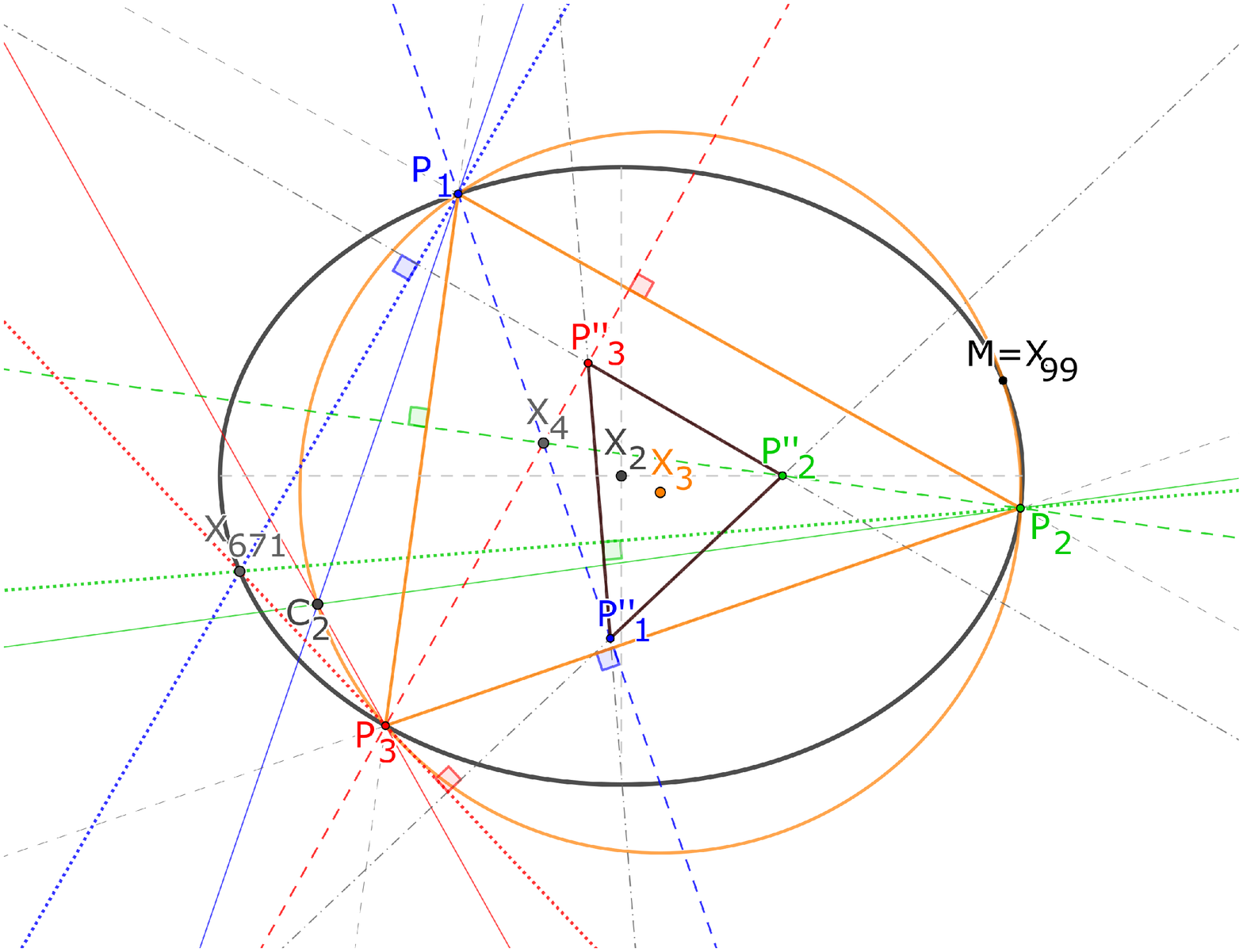}}
\end{minipage}
\caption{\textbf{Left:} $T$ and $T'$ are perspective on $X_2$. They are also orthologic, with orthology centers $X_{671}$ and the reflection of $X_{99}$ on $X_4$. \textbf{Right:} $T$ and $T''$ are perspective on $X_4$. They are also orthologic, with orthology centers $X_4$ and $X_{671}$.}
\label{fig:persp-t-tp}
\end{figure}

\begin{proof}
We denote by $X_{671}$ the reflection of $M=X_{99}$ on $O=X_2$. From Lemma~\ref{lem:perp-pi}, the line through $M$ and $-P_1$ is perpendicular to $P_2'' P_3''$. Reflecting about $O=X_2$, the line $P_1 X_{671}$ is also perpendicular to $P_2'' P_3''$. Since $P_2'' P_3'' \parallel P_2' P_3'$ from the homothety, we get that $P_1 X_{671}\perp P_2' P_3'$. This means the perpendicular from $P_1$ to $P_2' P_3'$ passes through $X_{671}$. Analogously, the perpendiculars from $P_2$ to $P_1' P_3'$ and from $P_3$ to $P_1' P_2'$ also go through $X_{671}$. Therefore $T$ and $T'$ are orthologic and $X_{671}$ is one of their two orthology centers.

Let $X_h$ be the reflection of $M$ on $X_4$. From Proposition~\ref{prop:perspTTpp}, the line through $P_1$ and $P_1''$ passes through the orthocenter $X_4$ of $T$, that is, $P_1''$ is on the $P_1$-altitude of $T$. This means that the perpendicular from $P_1''$ to the line $P_2 P_3$ passes through $X_4$. Applying the homothety with center $M$ and ratio 2, the perpendicular from $X_1'$ to $X_2 X_3$ passes through $X_h$. Analogously, the perpendiculars from $X_2'$ to $X_1 X_3$ and from $X_3'$ to $X_1 X_2$ also pass through $X_h$. Hence, $X_h$ is the second orthology center of $T$ and $T'$.
\end{proof}

\begin{theorem}
\label{th:orthologic2}
Triangles $T$ and $T''$ are orthologic and their centers of orthology are $X_4$ and the reflection  $X_{671}$   of $M$ on $X_2$.
\end{theorem}
\begin{proof}
From Proposition~\ref{prop:perspTTpp}, the perpendiculars from $P_1''$ to $P_2 P_3$, from $P_2''$ to $P_1 P_3$, and from $P_3''$ to $P_1 P_2$ all pass through $X_4$. Thus, triangles $T$ and $T''$ are orthologic and $X_4$ is one of their two orthology centers.

As before, we denote by $X_{671}$ the reflection of $M=X_{99}$ on $O=X_2$. Again, from Lemma~\ref{lem:perp-pi}, the line through $M$ and $-P_1$ is perpendicular to $P_2'' P_3''$, so reflecting it at $X_2$, we get that $P_1 X_{671}\perp P_2'' P_3''$. Since the triangles $T''$ and $T'$ have parallel sides, we get ${P_1}X_{671} \perp P_1'P_3' \parallel P_1''P_3''$. Thus, $X_{671}$ is the second orthology center of $T$ and $T''$. 
\end{proof}

\begin{theorem}[Sondat's Theorem]
If two triangles are both perspective and orthologic, their centers of orthology and perspectivity are collinear. Moreover, the line through these centers is perpendicular to the perspectrix of the two triangles \cite{thebault-1952,patrascu2020}.
\end{theorem}


\noindent Referring to Figure~\ref{fig:euler-perp}:

\begin{theorem}
The perspectrix of $T$ and $T'$ is perpendicular to the Euler Line of $T$.
\end{theorem}

\begin{proof}
Since $T$ and $T'$ are both orthologic and perspective from Corollary~\ref{perspTTp} and Theorem~\ref{th:orthologic}, by Sondat's Theorem, their perspectrix is perpendicular to the line through their orthology centers (reflections of $M$ at $X_2$ and $X_4$) and perspector ($X_2'=C_2=X_{98}=$reflection of $M$ at $X_3$). By applying a homothety of center $M$ and ratio 1/2, this last line is parallel to the line through $X_2$, $X_3$, and $X_4$, the Euler line of $T$. Therefore the perspectrix of $T$ and $T'$ is perpendicular to the Euler line of $T$.
%
\end{proof}

\begin{proposition}
The perspectrix of $T$ and $T''$ is perpendicular to the line $X_4 X_{671}$ (which is parallel to the line through $M$ and $X_{376}$, the reflection of $X_2$ at $X_3$).
\end{proposition}

\begin{proof}
Since $T$ and $T''$ are both orthologic and perspective from Proposition~\ref{prop:perspTTpp} and Theorem~\ref{th:orthologic2}, by Sondat's Theorem, their perspectrix is perpendicular to the line through their orthology centers $X_4$ and $X_{671}$. Reflecting this last line at $X_2$, we find that it is parallel to the line through $M$ and the reflection of $X_4$ at $X_2$, which is the same as the reflection of $X_2$ at $X_3$. 
\end{proof}

\begin{figure}
\includegraphics[height=.3\textheight,trim=60 0 40 30 ,clip]{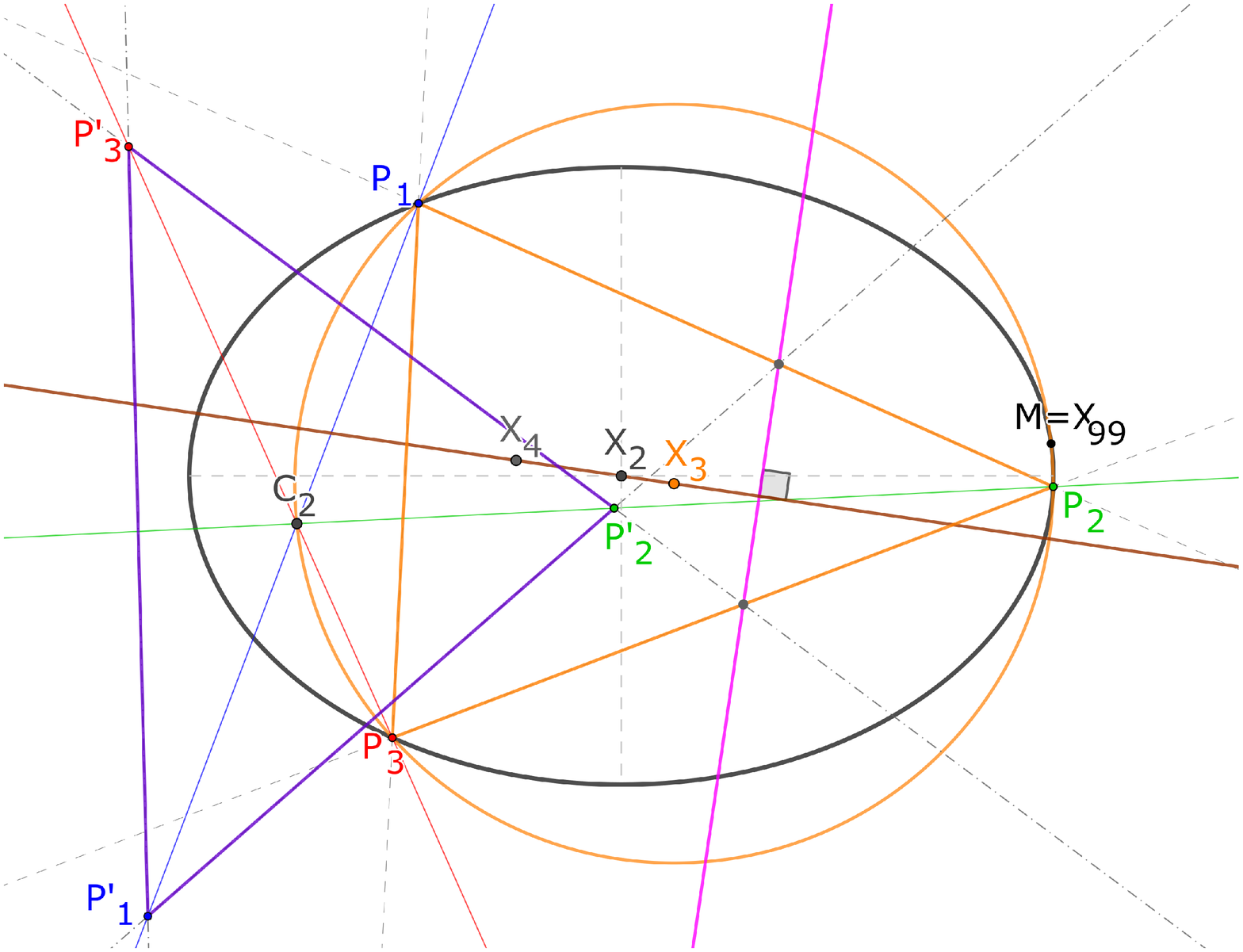}
\caption{The perspectrix of $T,T''$ is perpendicular to the line through $X_4$ and $X_{671}$. Compare with Figure~\ref{fig:non-transitive}: the perspectrix of $T,T'$ is perpendicular to the Euler Line of $T$.}
\label{fig:euler-perp}
\end{figure}

Table~\ref{tab:tri_ctrs} lists a few pairs of triangle centers numerically found to be common over $T,T'$ or $T,T''$.

\begin{table}[H]
\begin{tabular}{|c|l|l|}
\hline
$T$ & $T'$ & $T''$ \\
\hline
$X_3$ & -- & $X_2''$ \\
$X_4$ & -- & $X_{671}''$ \\
$X_5$ & -- & $X_{115}''$ \\
$X_{20}$ & -- & $X_{99}''$ \\
$X_{76}$ & -- & $X_{598}''$ \\
$X_{98}$ & $X_2'$ & -- \\
$X_{114}$ & $X_{230}'$ & -- \\
$X_{382}$ & -- & $X_{148}''$ \\
$X_{548}$ & -- & $X_{620}''$ \\
$X_{550}$ & -- & $X_{2482}''$ \\
\hline
\end{tabular}
\caption{Triangle Centers which coincide $T,T'$ or $T,T''$.}
\label{tab:tri_ctrs}
\end{table}

\section{Addendum: Rotated Negative Pedal Curve}
\label{sec:npc-rot}
The Negative Pedal Curve is the envelope of lines $L(t)$ passing through $P(t)$ and perpendicular to $P(t)-M$. Here we consider the envelope $\Delta^*_{\theta}$ of the $L(t)$ rotated clockwise a fixed $\theta$ about $P(t)$; see Figure~\ref{fig:npc-rot}.

\begin{figure}
    \centering
    \includegraphics[width=\textwidth]{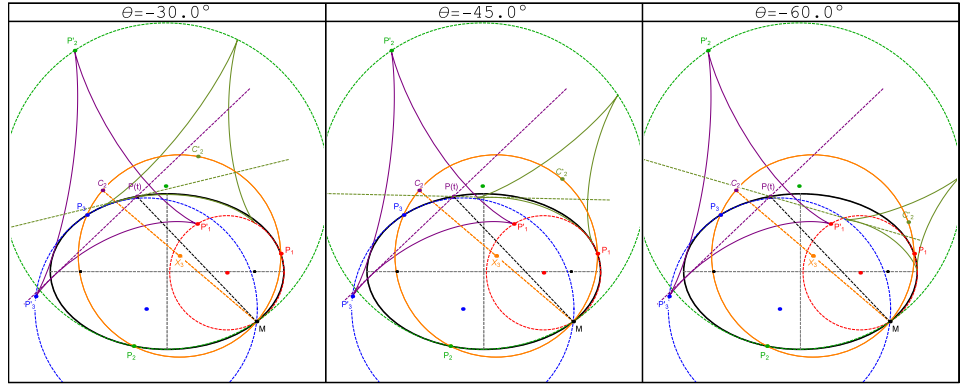}
    \caption{From left to right: for a fixed $M$, the line passing through $P(t)$ and perpendicular to the segment $P(t)-M$ (dashed) purple is rotated clockwise by $\theta=30,45,60$ degrees, respectively (dashed olive green). For all $P(t)$ these envelop new constant-area crooked hats $\Delta^*$ (olive green) whose areas are $\cos(\theta)^2$ that of $\Delta$. For the $\theta$ shown, these amount to $3/4,1/2,1/4$ of the area of $\Delta$ (purple). As one varies $\theta$, the center of area $C_{2}^*$ of $\Delta^*$ sweeps the circular arc between $C_2$ and $M$ with center at angle $2\theta$. The same holds for the cusps running along the corresponding osculating circles (shown dashed red, green, blue), which are stationary and independent of $\theta$.
    \textbf{Video:} \cite[PL\#06]{playlist2020-deltoid}}
    \label{fig:npc-rot}
\end{figure}

\begin{proposition}
$\Delta_{\theta}^\ast$ is the image of the NPC $\Delta$ under the similarity which is the product of a rotation about $M$ through $\theta$ and a homothety with center $M$ and factor $\cos\theta$.
\label{prop:bottema}
\end{proposition}
\begin{proof} 
For variable parameter $t$, the lines $L(t)$ and $L_{\theta}^\ast(t)$ perform a motion which sends $P$ along $\E$, while the line through $P$ orthogonal to $L(t)$ slides through the fixed point $M$. Due to basic results of planar kinematics \cite[p.~274]{Bottema1979}, the instantaneous center of rotation $I$ lies on the normal to $\E$ at $P$ and on the normal to $MP$ at $M$. We obtain a rectangle with vertices $P$, $M$ and $I$. The fourth vertex is the enveloping point $C$ of $L(t)$. The enveloping point $C^\ast$ of $L_{\theta}^\ast$ is the pedal point of $I$. Since the circumcircle of the rectangle with diameter $MC$ also passes through $C^\ast$, we see that $C^\ast$ is the image of $C$ under the stated similarity, Figure~\ref{fig:delta-theta-homothety}.

This holds for all points on $\Delta$, including the cusps, but also for the center $C_2$. At poses where $C$ reaches a cusp $P_i'$ of $\Delta$, then for all lines $L_{\theta}^\ast(t)$ through $P$ the point $C^\ast$ is a cusp of the corresponding envelope. Then the point is the so-called return pole, and the circular path of $C$ the return circle or cuspidal circle \cite[p.~274]{Bottema1979}. 
\end{proof}

\begin{figure}
    \centering
    \includegraphics[width=.7\textwidth]{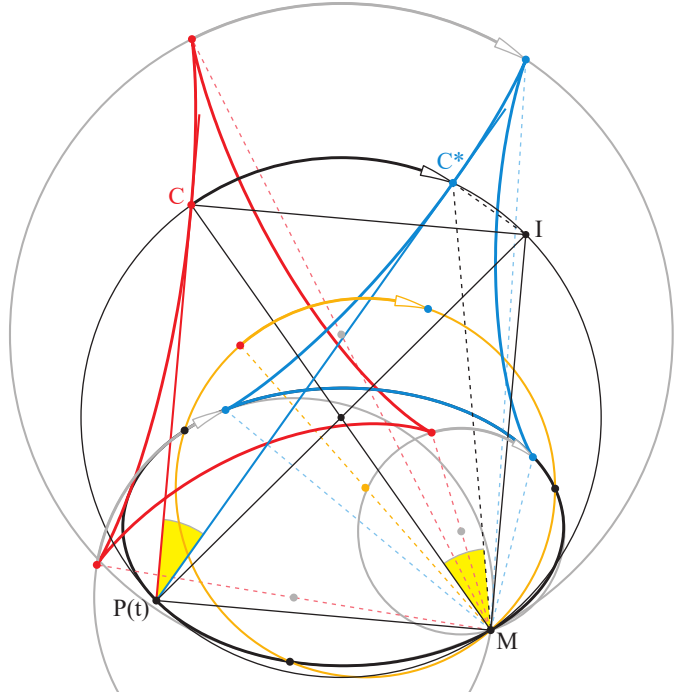}
    \caption{The construction of points $C$, $  C^* $of the envelopes $ \Delta$ (red) and $ \Delta^*$ (blue) with the help of the instant center of rotation $I$ reveals that the rotation about $M $ through $  \theta$ and scaling with factor $\cos\theta $ sends $\Delta$ to $\Delta^* $  (Proposition~
    \ref{prop:bottema}).}
    \label{fig:delta-theta-homothety}
\end{figure}

\begin{corollary} The area of $\Delta_{\theta}^*$ is independent of $M$ and is given by:

	\[A=\frac{c^4\cos^2\theta\; \pi}{2ab}. \]
\end{corollary}

Note this is equal to $\cos^2\theta$ of the area of $\Delta$, see Equation~\ref{eqn:adelta}. 

\begin{remark}
For variable $\theta$ between $-90^\circ$ and $90^\circ$, the said similarity defines an equiform motion where each point in the plane runs along a circle through $M$ with the same angular velocity.
For each point, the configuration at $\theta = 0$ and $M$ define a diameter of the trajectory. 
\end{remark}

Recall the pre-images $P_i$ of the cusps of $\Delta$ have vertices at $\E(t_i)$, where $t_1=-\frac{u}{3}$, $t_2=-\frac{u}{3}-\frac{2\pi}{3}$, $t_3=-\frac{u}{3}-\frac{4\pi}{3}$, see Theorem~\ref{th:circle_5}.

\begin{corollary}
The cusps $P_i^*$ of $\Delta_{\theta}^*$ have pre-images on $\E$ which are invariant over $\theta$ and are congruent with the $P_i,i=1,2,3$. 
\end{corollary}

\begin{corollary}
Lines ${P_i}{P_i^*}$ concur at $C_2^*$.
\end{corollary}

\begin{corollary}
$C_2^*$ is a rotation of $C_2$ by $2\theta$ about the center $X_3$ of $\K$. In particular, When $\theta=\pi/2$, $C_2^*=M$, and $\Delta_{\theta}^*$ degenerates to point $M$.
\end{corollary}



\section{Conclusion}
\label{sec:conclusion}
Before we part, we would like to pay homage to eminent swiss mathematician Jakob Steiner (1796--1863), discoverer of several concepts appearing herein: the Steiner Ellipse and Inellipse, the Steiner Curve (or hypocycloid), the Steiner Point $X_{99}$. Also due to him is the concept of orthologic triangles and the theorem of 3 concurrent osculating circles in the ellipse. Hats off and vielen dank, Herr Steiner!


Some of the above phenomena are illustrated dynamically through the videos on Table~\ref{tab:videos}.

\begin{table}[H]
\begin{tabular}{|c|l|l|l|}
\hline
\href{https://bit.ly/3gqfIRm}{PL\#} & Title & Narrated \\
\hline

\href{https://youtu.be/wetmchfY5jI}{01} & 
\makecell[lt]{Constant-Area Deltoid} & no \\

\href{https://youtu.be/LxADeM1-WHw}{02} & 
\makecell[lt]{Properties of the Deltoid} & yes \\

\href{https://youtu.be/NwXc-Vfjs98}{03} & 
\makecell[lt]{Osculating Circles at the Cusp Pre-Images} & yes \\
\href{https://youtu.be/rZht21KFXk4}{04} & 
\makecell[lt]{Loci of Cusps and $C_2$} & no \\
\href{https://youtu.be/fwyr6LXFS1c}{05} & 
\makecell[lt]{Concyclic pre-images, osculating circles,\\and 3 area-invariant triangles} & no \\
\href{https://youtu.be/DgADxkqlKSw}{06} & Rotated Negative Pedal Curve & yes \\
\hline
\end{tabular}
\caption{Playlist of videos. Column ``PL\#'' indicates the entry within the playlist.}
\label{tab:videos}
\end{table}

\appendix
\section{Explicit Expressions for the \texorpdfstring{$P_i,P_i',P_i''$}{Pi,Pi',Pi''}}

 {
\begin{align*}
P_1=&\left[ a\cos\frac{u}{3}, -b\sin\frac{u}{3} ,\right]\\
P_2=&\left[  -a\sin(\frac{u}{3}+\frac{\pi}{6}), -b\cos(\frac{u}{3}+\frac{\pi}{6})\right]\\
P_3
=&\left[  -a\cos(\frac{u}{3}+\frac{\pi}{6}), b\sin(\frac{u}{3}+\frac{\pi}{6}) \right]\\
P_1'=&\left[  
 \frac {3c^2}{2a}\cos \frac{u}{3} - \, \frac {
 \left( {a}^{2}+{b}^{2} \right)}{2a} \cos   u    , \frac {3{c}^{2}}{2b}\sin \frac{u}{3} - \, \frac {
 \left( {a}^{2}+{b}^{2} \right)}{2b} \sin u   
 \right]\\
 P_2'=& \left[ -  \frac {3{c}^{2}}{4a}\cos \frac{u}{3}    - \,
\frac {3\sqrt {3}{c}^{2}}{4a} \sin \frac{u}{3}- \,\frac {
 \left( {a}^{2}+{b}^{2} \right)}{2a} \cos   u 
,   \frac {3\sqrt {3}{c}^{2}}{4b}\cos \frac{u}{3}- \,
\frac {3{c}^{2}}{4b}\sin \frac{u}{3}- \,   \frac{\left( {a}^{2}+{b}^{2} \right) }{2b}\sin u
\right]\\
P_3'=&\left[- \, \frac {3{c}^{2}}{4a}\cos \frac{u}{3} + \,
\frac {3\sqrt {3}{c}^{2}}{4a}\sin\frac{u}{3} - \,\frac {
 \left( {a}^{2}+{b}^{2} \right) }{2a}\cos u  , -\frac {3\sqrt {3}{c}^{2}}{4b}\cos \frac{u}{3}- \,
\frac {3{c}^{2}}{4b}\sin \frac{u}{3}- \,   \frac{\left( {a}^{2}+{b}^{2} \right) }{2b}\sin u
\right]\\
P_1''=&\left[\frac{3c^2}{4a} \cos\frac{u}{3} +\frac{c^2}{4a}\cos u,\frac{3c^2}{4b} \sin\frac{u}{3} +\frac{c^2}{4b}\sin u
\right]\\
P_2''=&\left[-\frac{3c^2}{8a}\cos\frac{u}{3}  -\frac{3\sqrt{3}c^2}{8a}\sin\frac{u}{3}+\frac{c^2}{4a} \cos u ,\frac{3\sqrt{3}c^2}{8b }\cos\frac{u}{3} -\frac{3c^2}{8b}\sin\frac{u}{3} -\frac{c^2}{4b}\sin u
\right]\\
P_3''=&\left[-\frac{3c^2}{8a}\cos\frac{u}{3}  +\frac{3\sqrt{3}c^2}{8a}\sin\frac{u}{3}+\frac{c^2}{4a} \cos u,-\frac{3\sqrt{3}c^2}{8b }\cos\frac{u}{3} -\frac{3c^2}{8b}\sin\frac{u}{3} -\frac{c^2}{4b}\sin u
\right]
\end{align*}
}

\label{app:cusps}

\section{Table of Symbols}
\begin{table}[b]
\begin{tabular}{|c|l|l|}
\hline
symbol & meaning & note \\
\hline
$\E$ & main ellipse & \\
$a,b$ & major, minor semi-axes of $\E$ & \\
$c$ & half the focal length of $\E$ & $c^2=a^2-b^2$ \\
$O$ & center $\E$ & \\
$M,M_u$ & a fixed point on the boundary of $\E$ & \makecell[lt]{$[a\cos{u},b\sin{u}]$,\\ perspector of $T',T''$, $=X_{99}$} \\
$P(t)$ & a point which sweeps the boundary of $\E$ & $[a\cos{t},b\sin{t}]$\\
$L(t)$ & Line through $P(t)$ perp. to $P(t)-M$ & \\
$\Delta,\Delta_u$ & \makecell[lt]{Steiner's Hat, negative pedal curve\\of $\E$ with respect to $M$} & invariant area \\
$\Delta^*_{\theta}$ & envelope of $L(t)$ rotated $\theta$ about $P(t)$ & invariant area \\
$\bar{C}$ & average coordinates of $\Delta$ & $=C_2=X_2'$ \\
$C_2,C_2^*$ & area center of $\Delta,\Delta^*$ & $C_2=X_2'=X_{98}$ \\
$P_i',P_i,P_i''$ & \makecell[lt]{The cusps of $\Delta$, their pre-images,\\and centers of $K_i$ (see below)} & \\
$P_i^*$ & cusps of $\Delta^*$ & $P_i{P_i^*}$ concur at $C_2^*$\\
$T,T',T''$ & triangles defined by the $P_i,P_i',P_i''$ & invariant area over $M$ \\
$A,A',A''$ & areas of $T,T',T''$ & $A'/A''=4$ for any $M,a,b$ \\
\hline
$\S$ & Steiner's Curve & aka. Hypocycloid and Triscupoid \\
$\E'$ & \makecell[lt]{Steiner Circumellipse\\of cusp ($P_i'$) triangle} &  centered at $C_2$\\
$a',b'$ & major, minor semi-axes of $\E'$ & \makecell[lt]{invariant, axis-parallel\\and similar to $90^\circ$-rotated $\E$} \\
$\K$ & Circumcircle of $T$ & center $X_3$, contains $M,P_i,C_2,C_2^*$\\
$\K'$ & Circumcircle of $T'$ & \\
$\K_i$ & Circles osculating $\E$ at the $P_i$ & contain $P_i,P_i',M$\\
$\E^*$ & evolute of $\E$ & the $K_i$ lie on it \\
$\H$ & Apollonius Hyperbola of $\E$ wrt $M$ & $\Delta$ is tangent to $\E$ at $\H{\cap}\E$\\
\hline
$X_3$ & circumcenter of $T$ & $=X_2''$ \\
$X_4$ & perspector of $T$ and $T''$ & \\
$X_{99}$ & Steiner Point of $T$ & $=M$ \\
$X_{98}$ & Tarry Point of $T$ & $=C_2$ \\
$X_2'$ & centroid of $T'$ & $=C_2$, and perspector of $T,T'$ \\
$X_{99}'$ & Steiner Point of $T'$ & \\
$X_2''$ & centroid of $T''$ & $=X_3$ \\
\hline
\end{tabular}
\caption{All Symbols used.}
\label{tab:symbols}
\end{table}

\label{app:symbols}



\clearpage
\bibliographystyle{maa}
\bibliography{references}

\end{document}